\documentclass[12pt]{amsart}
\usepackage{amsmath}
\usepackage{amssymb}
\usepackage{amsfonts} 
\usepackage{mathrsfs}
\usepackage{mathtools}

  \newcommand{\Cdb}{\mbox{$\mathbb{C}$}}

   \newcommand{\Rdb}{\mbox{$\mathbb{R}$}}

   \renewcommand{\H}{\mbox{${\mathcal H}$}}

    \newcommand{\M}{\mbox{${\mathcal M}$}}
   \newcommand{\N}{\mbox{${\mathcal N}$}}

\newcommand{\norm}[1]{\Vert#1\Vert}
\newcommand{\bignorm}[1]{\bigl\Vert#1\bigr\Vert}
\newcommand{\Bignorm}[1]{\Bigl\Vert#1\Bigr\Vert}

\newtheorem{theorem}{Theorem}[section]
\newtheorem{lemma}[theorem]{Lemma}
\newtheorem{corollary}[theorem]{Corollary}
\newtheorem{proposition}[theorem]{Proposition}
\newtheorem{definition}[theorem]{Definition}
\theoremstyle{remark}
\newtheorem{remark}[theorem]{\bf Remark}
\theoremstyle{definition}

\numberwithin{equation}{section}

\author[C. Kriegler]{Christoph Kriegler}
\email{christoph.kriegler@uca.fr}
\address{Universit\'e Clermont Auvergne, 
CNRS, LMBP, F-63000 Clermont-Ferrand, France}
\author[C. Le Merdy]{Christian Le Merdy}
\email{clemerdy@univ-fcomte.fr}
\address{Laboratoire de Math\'ematiques de Besan\c con, UMR 6623,
CNRS, Universit\'e Marie et Louis Pasteur,
25030 Besan\c{c}on Cedex, France}
\author[S. Zadeh]{Safoura Zadeh}
\email{jsafoora@gmail.com}
\address{School of Mathematics, University of Bristol, United Kingdom}

\begin{document}

\title{Positive isometric Fourier multipliers on non-commutative $L^p$-spaces}

\begin{abstract}
For a locally compact group \(G\), let \(\mathcal{L}G\) denote its left group von Neumann algebra and let \(L^p(\mathcal{L}G)\), \(1 \le p \le \infty\), be the corresponding non-commutative \(L^p\)-space. Given \(\phi \in L^\infty(G)\), we study the Fourier multiplier \(M_{\phi,p}\) acting on \(L^p(\mathcal{L}G)\). We prove that for any \(p \neq 2\), the operator \(M_{\phi,p}\) is a positive surjective isometry if and only if \(\phi\) coincides locally almost everywhere with a continuous character of \(G\). This characterization extends results obtained recently (jointly with C.~Arhancet) in the unimodular setting.
\end{abstract}

\maketitle



\bigskip
\section{Introduction}\label{1}

\noindent Let $\Gamma$ be an abelian locally compact group. Parrott~\cite{Par} and Strichartz~\cite{St} showed that, for $p \neq 2$, a Fourier multiplier $T$ on $L^p(\Gamma)$ is an isometry if and only if there exist $c \in \mathbb{T}$, the unit circle, and $\gamma \in \Gamma$ such that \(T = c\,\tau_\gamma\), where $\tau_\gamma(f) = f(\cdot - \gamma)$ denotes translation by $\gamma$. Equivalently, the symbol of $T$ is a continuous character of the dual group $G = \widehat{\Gamma}$ multiplied by an element of $\mathbb{T}$. This characterization reveals a strong rigidity phenomenon for isometric Fourier multipliers in the commutative setting.

For non-abelian groups, harmonic analysis and Fourier multiplier theory naturally lead to the framework of left group von Neumann algebras and their associated non-commutative $L^p$-spaces. Let $G$ be a locally compact group equipped with a left Haar measure, and let $\mathcal L G\subset B(L^2(G))$ denote its left group von Neumann algebra. This algebra carries a canonical normal semifinite faithful weight, the left Plancherel weight. When $G$ is abelian, $\mathcal L G$ identifies with the $L^\infty$-space over the dual group, and the Plancherel weight coincides with integration against the Haar measure on the dual group. The Plancherel weight is a trace if and only if $G$ is unimodular, in which case the $L^p$ theory simplifies considerably.

The study of Fourier multipliers on non-commutative $L^p$-spaces has attracted considerable attention in recent years, largely within the unimodular setting; see, for instance, \cite{CDLS, GJP, PRD}, as well as \cite{P} and the references therein. In joint work with C.~Arhancet~\cite{AKLZ}, we contributed to this line of research by extending the Parrott--Strichartz characterization of isometric Fourier multipliers to $L^p(\mathcal L G)$ in the unimodular/tracial case. The aim of the present paper is to push this analysis beyond the unimodular framework. In the non-unimodular setting, the modular function introduces an intrinsic asymmetry between left and right structures, and the resulting absence of a trace becomes a principal source of new difficulties. Consequently, several arguments from~\cite{AKLZ} no longer apply and must be replaced by new techniques tailored to the setting of non-tracial weights.

Non-unimodular locally compact groups arise naturally in many fundamental examples, including affine groups of the form $\mathbb{R} \rtimes \mathbb{R}^+$ and, more generally, groups of $ax+b$ type. Only recently has a systematic study of $L^p$-Fourier multipliers in this setting begun to emerge; see \cite{Cas, CPPR, CDLS, DL, PDT, Terp2, VosPaper, Vos}.

There are several equivalent constructions of non-commutative $L^p$-spaces in the non-tracial setting, most notably those due to Haagerup, Connes--Hilsum, and Kosaki (see \cite{H, Hil, Hiai, Ko, Terp0, Terp1}). These frameworks provide natural extensions of classical $L^p$-spaces to the non-commutative context. In this paper, we investigate Fourier multipliers on $L^p(\mathcal L G)$ using the Connes--Hilsum construction.

In Section~\ref{3}, we recall the definition of bounded Fourier multipliers on $L^p(\mathcal L G)$ for $1 \le p < \infty$, following the approach of \cite{Vos}, and show that the various natural formulations of this notion are equivalent. Particular attention is devoted to the range $p < 2$, which is typically addressed only briefly in the existing literature on Fourier multipliers for non-unimodular groups. Our aim is to fill this gap.

Let $1\le p<\infty$ and $\phi\in L^\infty(G)$. We say that $\phi$ is a bounded Fourier multiplier on $L^p(\mathcal L G)$ if there exists $C\ge0$ such that for all $f\in\mathcal H\ast\mathcal H$,
\[
\Delta^{\frac1{2p}}\lambda(\phi f)\Delta^{\frac1{2p}}\in L^p(\mathcal L G)
\quad\text{and}\quad
\bigl\|\Delta^{\frac1{2p}}\lambda(\phi f)\Delta^{\frac1{2p}}\bigr\|_p
\le
C\,\bigl\|\Delta^{\frac1{2p}}\lambda(f)\Delta^{\frac1{2p}}\bigr\|_p.
\]
Here, $\mathcal H=\{f\in L^2(G):\operatorname{Supp}(f)\text{ is compact}\}$ and $\Delta$ denotes the operator of multiplication by the modular function of $G$. In this case, the map
\[
\Delta^{\frac1{2p}}\lambda(f)\Delta^{\frac1{2p}}
\longmapsto
\Delta^{\frac1{2p}}\lambda(\phi f)\Delta^{\frac1{2p}},\qquad f\in\mathcal H\ast\mathcal H,
\]
extends uniquely to a bounded operator \(M_{\phi,p}\colon L^p(\mathcal L G)\to L^p(\mathcal L G)\).

In Section~\ref{4}, we show that $\phi$ is a bounded Fourier multiplier on $L^p(\mathcal L G)$ if and only if it is a bounded Fourier multiplier on $L^{p'}(\mathcal L G)$, where $p^{-1} + {p'}^{-1} = 1$. We also prove that continuous characters of $G$ give rise to positive onto isometric Fourier multipliers on $L^p(\mathcal L G)$ for all $p$. Our main result, Theorem~\ref{5Positive}, establishes a converse for $p \neq 2$: if $M_{\phi,p}$ is a positive Fourier multiplier that is an onto isometry, then $\phi$ coincides locally almost everywhere with a continuous character of $G$.

Section~\ref{6} is devoted to the endpoint cases $p=1$ and $p=2$. We show that the rigidity phenomenon established for $1<p\neq 2<\infty$ continues to hold at $p=1$ without any positivity assumption, and at $p=2$ under a slightly stronger condition.

\section{Group von Neumann algebras and their non-commutative $L^p$-spaces}
\subsection{The weighted von Neumann algebras $({\mathcal L}G,\varphi_0)$ and $({\mathcal R}G,\psi_0)$}

We briefly recall the definition of the left and right group von Neumann algebras, starting with some required background on locally compact groups; for details, see \cite{Fo}, \cite[Section 2]{Haag}, \cite[Chapter 18]{Stra}  and \cite[Section VII.3]{Tak2}. 

Let G be a locally compact group equipped with a fixed left Haar measure dt. For $1 \leq p < \infty$, we write $\norm{\cdot}_p$ for the norm on the associated $L^p$-space $L^p(G)$.

Given any two locally measurable functions $\phi_1,\phi_2\colon G \to \Cdb$, we say that $\phi_1=\phi_2$ locally almost everywhere if the set $\{\phi_1\not=\phi_2\}$ is locally null. Let $L^\infty(G)$ denote the space of locally measurable functions $\phi \colon G \to \Cdb$ that are bounded except on a locally null set, modulo functions that are zero locally almost everywhere. Equipped with
\[
\norm{\phi}_{\infty} = \inf \{ c \,:\, |\phi(t)| \le c \text{ locally a.e.} \},
\]
this is a Banach space. For any $\phi\in L^\infty(G)$ and $f\in L^1(G)$, the product $\phi f\in L^1(G)$ is well-defined. Moreover, the duality pairing $\langle\phi,f\rangle=\int_G\phi f\,$ yields an isometric identification $L^\infty(G)=L^1(G)^*$. When $G$ is $\sigma$-compact, $L^\infty(G)$ coincides with the usual Lebesgue $L^\infty$-space associated with the Haar measure, although the two spaces may differ in general. See \cite[Section 2.3]{Fo} for details, and also \cite[Lemma 3.2]{AKLZ} and the subsequent discussion. Within $L^\infty(G)$ we distinguish two important subspaces. The subspace $C_b(G)$ consists of all bounded continuous functions on $G$, and the subspace $C_c(G)$ consists of compactly supported continuous functions.

Finally, we recall the modular function $\Delta \colon G \to (0,\infty)$. 
This is a continuous group homomorphism defined by the way the left Haar measure transforms under right translation and inversion. More precisely, we have
\[
\int_G f(t)\, dt = \int_G \Delta(t)^{-1} f(t^{-1})\, dt, \qquad f \in L^1(G),
\]
and
\[
\int_G f(ts)\, dt = \Delta(s)^{-1} \int_G f(t)\, dt, \qquad f \in L^1(G).
\]
These two properties will be used freely in the sequel.

Next, we recall the definition of the regular representations of a locally compact group. The left and right representations 
$\lambda, \rho \colon G \to B(L^2(G))$ are defined by
\[
[\lambda(t)h](s) = h(t^{-1}s), \qquad
\hbox{and}\qquad
[\rho(t)h](s) = \Delta(t)^{\frac12} h(st),
\]
for any $t,s \in G$ and $h \in L^2(G)$. 

Let ${\mathcal L}G \subset B(L^2(G))$ and ${\mathcal R}G \subset B(L^2(G))$ denote the von Neumann algebras generated by $\lambda(G)$ and $\rho(G)$, respectively.  
These are called the left and right group von Neumann algebras of \(G\) and they are mutual commutants, that is,
\begin{equation}\label{2Comm}
\mathcal{R}G = \mathcal{L}G'.
\end{equation}

The algebra \(\mathcal{L}G\) admits a concrete description via convolution operators. Indeed, given $f \in L^1(G)$, we define
\[
\lambda(f) = \int_G f(t) \lambda(t)\, dt,
\]
where the integral is understood in the strong sense.  
The resulting operator acts on \(L^2(G)\) by left convolution, that is, $[\lambda(f)](h)(s) = (f \star h)(s) = \int_G f(t)h(t^{-1}s) \, dt$ for all $h \in L^2(G)$.  
Moreover, $\lambda(f)$ belongs to ${\mathcal L}G$, and satisfies   
$\norm{\lambda(f)} \leq \norm{f}_1$.  
The image $\lambda(L^1(G))$ is $w^*$-dense in ${\mathcal L}G$ 
and since $C_c(G) \subset L^1(G)$ is a 
dense subspace, the same is true for $\lambda(C_c(G))$.

The map \(f\mapsto\lambda(f)\) respects both the algebraic and involutive structures. Indeed, $\lambda(f_1\star f_2)=\lambda(f_1)\lambda(f_2)$ for all $f_1,f_2\in L^1(G)$. If, for
any $f\in L^1(G)$, \(f^*\) is defined by
$$
f^*(t) = \Delta(t)^{-1}\overline{f(t^{-1})},
$$
then $f^*\in L^1(G)$ and $\lambda(f)^*=\lambda(f^*)$. Thus, $\lambda(L^1(G))$ is a $*$-subalgebra of ${\mathcal L}G$.

Associated with \(\mathcal{L}G\) is the left Plancherel weight 
$
\varphi_0\colon {\mathcal L}G^+\longrightarrow [0,\infty].
$
This weight is normal, semifinite and faithful, and 
is characterized by the identity
\begin{equation}\label{2FP}
\varphi_0\bigl(\lambda(f)^*\lambda(f)
\bigr) = \norm{f}_{2}^2,\qquad f\in L^1(G)\cap L^2(G).
\end{equation}
The modular nature on the group is reflected in tracial properties of \(\varphi_0\): $\varphi_0$ is a trace if and only if $G$ is unimodular (that is, $\Delta\equiv 1$).

A completely analogous construction applies to the right von Neumann algebra \(\mathcal{R}G\). The right Plancherel weight \(\psi_0\) is normal, semifinite, faithful on \(\mathcal{R}G\) satisfying

$$
\psi_0\bigl(\rho(f)^*\rho(f)
\bigr) = \norm{f}_{2}^2,\qquad f\in L^1(G)\cap L^2(G).
$$
The two Plancherel weights are related by the standard form of the group von Neumann algebra (see e.g. \cite[Chapter 3]{Hiai}). Indeed,
let ${\mathfrak J}\colon L^2(G)\to L^2(G)$ be the 
anti-unitary operator on \(L^2(G)\) defined by 
$[{\mathfrak J}(h)](s)= \Delta(s)^{-\frac12}\overline{h(s^{-1})}$.  This operator intertwines the left and right regular representations in the sense that ${\mathfrak J}\lambda(t) {\mathfrak J}=\rho(t)$ for all $t\in G$. As a consequence, the weights \(\varphi_0\) and \(\psi_0\) are related by conjugation with \({\mathfrak J}\) as follows
\begin{equation}\label{2Conj}
\psi_0(a) = \varphi_0({\mathfrak J}a{\mathfrak J}),\qquad a\in {\mathcal R}G^+.
\end{equation}

\smallskip
\subsection{The spaces $L^p({\mathcal L}G)$}

We briefly recall the Connes–Hilsum construction of non-commu-\-tative $L^p$-spaces in a form adapted to group von Neumann algebras. 
For the general construction and its basic properties, as well as its relation to the Haagerup non-commutative $L^p$-spaces 
defined in \cite{H}, we refer to \cite[Section 11.3]{Hiai} or to the original paper \cite{Hil}. Additional background and related results can be found in \cite{Terp0} and \cite[Section 1]{HJX}.

For $1 \leq p \leq \infty$, we denote by $L^p({\mathcal L}G)$ the Connes–Hilsum non-commutative $L^p$-space associated with the triple
\( ({\mathcal L}G, L^2(G), \psi_0).
\)
These spaces are studied in detail in \cite{Terp2}. In what follows, we only recall the properties that will be used later and refer to \cite{Terp2} for proofs and further details (see also \cite[Chapter 11]{Hiai} and \cite{Vos}).

Throughout the sequel, we will identify $\Delta$ with the multiplication operator 
\[
h \mapsto \Delta h
\] 
on $L^2(G)$. This is a closed and densely defined operator with domain 
\[
\{ h \in L^2(G) \,:\, \Delta h \in L^2(G) \}.
\] 
The same convention applies to any measurable function $G \to \Cdb$ in place of $\Delta$, and in particular to the functions $\Delta^z$ for any $z \in \Cdb$.

Given a normal semifinite weight $\varphi$ on $\mathcal{L}G$, we denote by 
\(
\frac{d\varphi}{d\psi_0}
\)
its spatial derivative with respect to $\psi_0$ (see, e.g., \cite[Definition 11.14]{Hiai}). This is a positive operator on $L^2(G)$. As shown in \cite[p. 551]{Terp2}, the spatial derivative of the left Plancherel weight with respect to the right Plancherel weight is given by
\begin{equation}\label{2Standard}
\frac{d\varphi_0}{d\psi_0} = \Delta.
\end{equation}
This identity may be also deduced from either \cite[Lemma 10]{Con} or \cite[Proposition 11.6]{Hiai}, using (\ref{2Conj}).

Let $A_0$ be the space of closed and densely defined operators on $L^2(G)$ that are affiliated with ${\mathcal L}G$. 
For a real number $\alpha \neq 0$, let $A_\alpha$ be the set of all closed and densely defined operators $a$ on $L^2(G)$ such that, if $a = u |a|$ is the polar decomposition of $a$, then $u \in {\mathcal L}G$ and there exists a normal semifinite weight $\varphi$ on ${\mathcal L}G$ satisfying
\begin{equation}\label{2A}
|a| = \Bigl( \frac{d\varphi}{d\psi_0} \Bigr)^\alpha.
\end{equation}
In this case, the weight $\varphi$ is uniquely determined by \(a\).

Each $A_\alpha$ is a complex vector space under the strong sum of operators and the usual scalar multiplication. 
It coincides with the so-called $(-\alpha)$-homogeneous operators (see \cite[Section 2]{Terp2} or \cite[pp.~202--203]{Hiai}) and is closed under taking adjoints. Moreover, the family \((A_\alpha)_{\alpha\in\mathbb{R}}\) is compatible with multiplication: if $a \in A_\alpha$, $b \in A_\beta$, the operator $ab$, with domain
\[
\{ h \in \mathrm{dom}(b) \,:\, b(h) \in \mathrm{dom}(a) \},
\] 
is closable, and its closure belongs to $A_{\alpha+\beta}$ (see \cite[Remark 2.2]{Terp2} or \cite[Lemma 2.4.2]{Vos}) and 
moreover, this operation is associative. The closure of $ab$ is called the strong product of $a$ and $b$, and by convention it will simply be denoted by $ab$ in the sequel. 

By (\ref{2Standard}), $\Delta$ belongs to $A_1$.  
Hence, for any $z \in \Cdb$, we have $\Delta^z \in A_{\mathrm{Re}(z)}$.  
Therefore, for any $f \in L^1(G)$ and $z_1, z_2 \in \Cdb$, the operator
\begin{equation}\label{2Strong}
\Delta^{z_1} \lambda(f) \Delta^{z_2} \in A_{\mathrm{Re}(z_1) + \mathrm{Re}(z_2)}
\end{equation}
is well-defined.

Let ${\mathcal L}G_*$ denote the predual of ${\mathcal L}G$, and let $1 \leq p < \infty$.  
By definition, the non-commutative $L^p$-space $L^p({\mathcal L}G)$ is the set of all $a \in A_{\frac{1}{p}}$ such that (\ref{2A}) holds for some $\varphi \in {\mathcal L}G_*^+$.  
In this case, we set
\[
\norm{a}_p = [\varphi(1)]^{\frac{1}{p}}.
\]  
Equipped with this norm, $L^p({\mathcal L}G)$ becomes a Banach space and is naturally contained in $A_p$. 
Moreover, $a^* \in L^p({\mathcal L}G)$ for all $a \in L^p({\mathcal L}G)$ and satisfies
\[
\norm{a^*}_p = \norm{a}_p.
\]  
By convention, $L^\infty({\mathcal L}G) = {\mathcal L}G$.

The positive cone $L^p({\mathcal L}G)^+$ consists by definition of those $a \in L^p({\mathcal L}G)$ that are positive in the sense of (closed and densely defined) operators over a Hilbert space, and a bounded operator $S \colon L^p({\mathcal L}G) \to L^p({\mathcal L}G)$ is called \emph{positive} if it maps $L^p({\mathcal L}G)^+$ into $L^p({\mathcal L}G)^+$.

For any $1 \leq p,q,r \leq \infty$ with $p^{-1} + q^{-1} = r^{-1}$, the product mapping
\[
A_{\frac{1}{p}} \times A_{\frac{1}{q}} \longrightarrow A_{\frac{1}{r}}
\]
arising from the above discussion restricts to a contractive bilinear map
\[
L^p({\mathcal L}G) \times L^q({\mathcal L}G) \longrightarrow L^r({\mathcal L}G).
\]
In particular, taking $q = \infty$ shows that, for every $p \geq 1$, the space
$L^p({\mathcal L}G)$ is a contractive ${\mathcal L}G$-bimodule.

The correspondence
\[
\varphi \in {\mathcal L}G_*^+ \longmapsto \frac{d\varphi}{d\psi_0} \in L^1({\mathcal L}G)^+
\]
extends to an isometric ${\mathcal L}G$-bimodule isomorphism \(\tau \colon {\mathcal L}G_* \to L^1({\mathcal L}G)\). This identification allows one to define a contractive linear functional
\[
\mathrm{Tr} \colon L^1({\mathcal L}G) \longrightarrow \Cdb
\]
by setting $\mathrm{Tr}(\tau(\varphi)) = \varphi(1)$ for every $\varphi \in {\mathcal L}G_*$. This functional is usually called the integral with respect to $\psi_0$ (see \cite[p.~203]{Hiai} or \cite[p.~554]{Terp2}) and provides the appropriate replacement for a trace in the non-unimodular setting.

Let $p' = p/(p-1)$ denote the conjugate exponent of $p$. Using this functional, one obtains a natural duality pairing between \(L^p({\mathcal L}G)\) and \(L^{p'}({\mathcal L}G)\) given by
\[
\langle a, b \rangle = \mathrm{Tr}(ab), \qquad 
a \in L^p({\mathcal L}G),\ b \in L^{p'}({\mathcal L}G).
\]
This pairing identifies \(L^{p'}({\mathcal L}G)\) isometrically with the dual space of \(L^p({\mathcal L}G)\).

Finally, the space $L^2({\mathcal L}G)$ carries a Hilbert space structure with inner product
\[
(a \mid b) = \mathrm{Tr}(b^* a), \qquad a,b \in L^2({\mathcal L}G).
\]
\smallskip
\subsection{Some dense subspaces}
In this section, we introduce certain dense subspaces of $L^p(\mathcal{L}G)$ which form the foundation for defining bounded Fourier multipliers. We first recall the subspaces of $\mathcal{L}G$ commonly used in connection with the left Plancherel weight $\varphi_0$. Let
\begin{equation}\label{3M}
\N = \{ x \in {\mathcal L}G \,:\, \varphi_0(x^*x) < \infty \}
\qquad\text{and}\qquad
\M = \mathrm{Span}\{ y^* x \,:\, x,y \in \N \}.
\end{equation}
It is well-known that $\N$ is a $w^*$-dense left ideal of ${\mathcal L}G$, while
$\M$ is a $w^*$-dense $*$-subalgebra of ${\mathcal L}G$; see, for example,
\cite[Section VII.1]{Tak2}.  Applying \cite[Theorem 26]{Terp1}, which holds for general Connes-Hilsum $L^p$-spaces, we obtain the following result.

\begin{lemma}\label{2Inc}
Let $1 \leq p < \infty$.
\begin{itemize}
\item[(1)] If $p \geq 2$, then $x \Delta^{\frac{1}{p}} \in L^p({\mathcal L}G)$ for all $x \in \N$.
\item[(2)] For any $p$, the product $\Delta^{\frac{1}{2p}} x \Delta^{\frac{1}{2p}}$ belongs to 
$L^p({\mathcal L}G)$ for all $x \in \M$.
\end{itemize}
\end{lemma}

In the special case $p = 2$, we further have 
$\lambda(L^1(G) \cap L^2(G)) \subset \N$ and
\begin{equation}\label{2FP}
\norm{\lambda(f)\Delta^{\frac12}}_2 = \norm{f}_2, \qquad 
f \in L^1(G) \cap L^2(G),
\end{equation}
see \cite[Theorem 3.2]{Terp2}.

Let $C_c(G) \star C_c(G)$ denote the linear span of all convolutions $f \star g$ with $f,g \in C_c(G)$.  
It follows from Lemma~\ref{2Inc} that for all \(1 \leq p < \infty\),
\[
\Delta^{\frac{1}{2p}} \, \lambda(C_c(G) \star C_c(G)) \, \Delta^{\frac{1}{2p}} \subset L^p(\mathcal{L}G),
\]
and for all \(2 \leq p < \infty\),
\[
\lambda(C_c(G)) \, \Delta^{\frac{1}{p}} \subset L^p(\mathcal{L}G).
\]

The next lemma establishes that these subspaces are in fact dense.

\begin{lemma}\label{2Dense}
Let $1 \leq p < \infty$.
\begin{itemize}
\item[(1)] If $p \geq 2$, then $\lambda(C_c(G)) \Delta^{\frac{1}{p}}$ is a dense subspace of
$L^p({\mathcal L}G)$.
\item[(2)] For any $p$, the space
\[
\Delta^{\frac{1}{2p}} \lambda(C_c(G) \star C_c(G)) \Delta^{\frac{1}{2p}}
\]
is dense in $L^p({\mathcal L}G)$.
\end{itemize}
\end{lemma}

\begin{proof}
Assertion~(1) follows from the proof of \cite[Theorem~4.5, (2)]{Terp2}.  
Assertion~(2) is obtained by applying~(1) with the index $2p$, and then using the contractive
product
\[
L^{2p}({\mathcal L}G) \times L^{2p}({\mathcal L}G) \longrightarrow L^p({\mathcal L}G).
\qedhere\]
\end{proof}

\medskip
\section{Bounded Fourier multipliers on $L^p({\mathcal L}G)$}\label{3}

In this section, we introduce bounded Fourier multipliers on $L^p(\mathcal{L}G)$ and show that all natural definitions of these multipliers are equivalent. 
Our approach relies on the modular structure associated with the left Plancherel weight $\varphi_0$ and on the identification of a convenient dense subspace of $L^p(\mathcal{L}G)$ on which Fourier multipliers can be defined in a canonical way. 
This subspace consists of operators arising from compactly supported functions and enjoys good stability and analyticity properties with respect to the modular action. 

Let $(\sigma_t)_{t\in{\mathbb R}}$ be the modular 
automorphism group of the weight  $\varphi_0$. 
Then 
$$
\sigma_t(x) = \Delta^{it}x\Delta^{-it}
$$
for all $x\in{\mathcal L}G$ and all $t\in\Rdb$, see \cite[Section 2]{Haag}.
This implies that
$$
\sigma_t(\lambda(f)) =\lambda(\Delta^{it}f),\qquad f\in L^1(G),\, t\in\Rdb.
$$
We denote by ${\mathcal L}G_a$ the set of all $x\in{\mathcal L}G$ such that $t\mapsto\sigma_t(x)$
extends to an entire function on $\Cdb$.
Such elements are called analytic, see e.g. \cite[Section VIII.2]{Tak2}.

Next, we introduce the space
$$
\H= \bigl\{f\in L^2(G)\, :\, {\rm Supp}(f)\ \hbox{is compact}\bigr\}.
$$
Clearly, $\H\subset L^1(G)\cap L^2(G)$, so $\lambda(f)$ is well-defined 
for $f\in\H$.
Moreover, for any 
$z\in\Cdb$ we have $\Delta^{z}f\in\H$. Hence $\lambda(f)$ 
belongs to ${\mathcal L}G_a$ and we have
$$
\sigma_{-iz}(\lambda(f)) = \lambda(\Delta^{z}f),\qquad f\in\H,\, z\in\Cdb.
$$
It therefore follows from the above that
\begin{equation}\label{3Comm}
\lambda(\Delta^z f) = \Delta^z\lambda(f)\Delta^{-z},\qquad f\in\H,\, z\in\Cdb.
\end{equation}
We also observe that
\begin{equation}\label{3HH}
\H\star\H\subset\H
\qquad\hbox{and}\qquad 
L^\infty(G)\,\cdotp \H=\H.
\end{equation}

For any $1\leq p<\infty$, we introduce
$$
{\mathcal B}_p = \Delta^{\frac{1}{2p}} \lambda(\H\star \H) \Delta^{\frac{1}{2p}}.
$$
By Lemmas \ref{2Inc} and \ref{2Dense}, this is a dense subspace of $L^p({\mathcal L}G)$.

\begin{definition}\label{3BFM} 
Let $\phi\in L^\infty(G)$ and let $1\leq p<\infty$. We say that $\phi$ is a bounded Fourier multiplier 
on $L^p({\mathcal L}G)$ if there exists a constant $C\geq 0$ such that
for all $f\in\H\star\H$,
$$
\Delta^{\frac{1}{2p}} \lambda(\phi f) \Delta^{\frac{1}{2p}} \in L^p({\mathcal L}G)
\qquad\hbox{and}\qquad
\bignorm{\Delta^{\frac{1}{2p}} \lambda(\phi f) \Delta^{\frac{1}{2p}}}_p\leq 
C\bignorm{\Delta^{\frac{1}{2p}} \lambda(f) \Delta^{\frac{1}{2p}}}_p.
$$
\end{definition}

Note that by (\ref{2Strong}), we know that 
$\Delta^{\frac{1}{2p}} \lambda(\phi f) \Delta^{\frac{1}{2p}}$ belongs
to $A_p$ for any $f\in\H\star \H$. However, there is no a priori reason for it to belong to
$L^p({\mathcal L}G)$.

If $\phi$ is a bounded Fourier multiplier 
on $L^p({\mathcal L}G)$, the mapping
$$
\Delta^{\frac{1}{2p}} \lambda(f) \Delta^{\frac{1}{2p}}\mapsto
\Delta^{\frac{1}{2p}} \lambda(\phi f) \Delta^{\frac{1}{2p}},\qquad f\in \H,
$$
uniquely extends to a bounded operator
$
M_{\phi,p}\colon L^p({\mathcal L}G)\longrightarrow L^p({\mathcal L}G).
$

We now collect a series of remarks showing that all reasonable formulations of bounded Fourier multipliers on $L^p({\mathcal L}G)$ are equivalent.

\begin{remark}\label{3Theta}
Consider $\phi\in L^\infty(G)$, $1\le p<\infty$, and $\theta\in[0,1]$. 
We show that $\phi$ is a bounded Fourier multiplier on $L^p({\mathcal L}G)$
if and only if there exists a constant $C\ge 0$ such that, for all
$f\in\H\star\H$,
\begin{equation}\label{3Alt}
\Delta^{\frac{\theta}{p}} \lambda(\phi f) \Delta^{\frac{1-\theta}{p}}
\in L^p({\mathcal L}G)
\quad\text{and}\quad
\bigl\|
\Delta^{\frac{\theta}{p}} \lambda(\phi f) \Delta^{\frac{1-\theta}{p}}
\bigr\|_p
\le C
\bigl\|
\Delta^{\frac{\theta}{p}} \lambda(f) \Delta^{\frac{1-\theta}{p}}
\bigr\|_p.
\end{equation}

Assume that $\theta\in\bigl[\frac12,1\bigr]$ and let
$c\ge 0$ satisfy $\frac{\theta}{p}=\frac{1}{2p}+c$. 
For any $f\in\H\star\H$, using~(\ref{3Comm}) we obtain
\[
\Delta^{\frac{\theta}{p}}\lambda(f)\Delta^{\frac{1-\theta}{p}}
=
\Delta^{\frac{1}{2p}}\lambda(\Delta^c f)\Delta^{\frac{1}{2p}}.
\]
Since $\Delta$ is multiplicative, 
multiplication by $\Delta^c$ maps
$\H\star\H$ onto itself, and therefore
\[
{\mathcal B}_p
=
\Delta^{\frac{\theta}{p}} \lambda(\H\star\H) \Delta^{\frac{1-\theta}{p}}.
\]

Similarly, by~(\ref{3HH}) and~(\ref{3Comm}), for all $f\in\H\star\H$ we have
\[
\Delta^{\frac{\theta}{p}}\lambda(\phi f)\Delta^{\frac{1-\theta}{p}}
=
\Delta^{\frac{1}{2p}}\lambda(\phi\Delta^c f)\Delta^{\frac{1}{2p}}.
\]
It follows that the  inequality in (\ref{3Alt})
is equivalent to the boundedness of the map
\[
\Delta^{\frac{1}{2p}}\lambda(f)\Delta^{\frac{1}{2p}}
\longmapsto
\Delta^{\frac{1}{2p}}\lambda(\phi f)\Delta^{\frac{1}{2p}}
\]
on the dense subspace $\Delta^{\frac{1}{2p}}\lambda(\H\star\H)\Delta^{\frac{1}{2p}}$
of $L^p({\mathcal L}G)$, which is precisely the definition of $\phi$ being
a bounded Fourier multiplier on $L^p({\mathcal L}G)$. In this case,
\[
M_{\phi,p}\Bigl(
\Delta^{\frac{\theta}{p}} \lambda(f) \Delta^{\frac{1-\theta}{p}}
\Bigr)
=
\Delta^{\frac{\theta}{p}} \lambda(\phi f) \Delta^{\frac{1-\theta}{p}},
\qquad f\in\H\star\H.
\]

An analogous argument applies when $\theta\in\bigl[0,\frac12\bigr]$,
writing $\frac{1-\theta}{p}=\frac{1}{2p}+c$.
\end{remark}

In the case $p=2$, the situation simplifies considerably, and we have the following.

\begin{corollary}\label{3Always}
Let $\phi\in L^\infty(G)$. Then $\phi$ is a bounded Fourier multiplier on $L^2({\mathcal L}G)$ and moreover,
$\norm{M_{\phi,2}}= \norm{\phi}_\infty.$
\end{corollary}

\begin{proof}
Let $f\in\H\star \H$. Using (\ref{2FP}), we have
$\norm{\lambda(f)\Delta^\frac12}_2=\norm{f}_2$ and
$\norm{\lambda(\phi f)\Delta^\frac12}_2=\norm{\phi f}_2$.  
The result then follows from Remark \ref{3Theta} applied with $p=2$ and $\theta=0$.
\end{proof}

\begin{remark}\label{3p-geq-2}
Let $\phi\in L^\infty(G)$ and let $2\leq p<\infty$.

\smallskip
{\bf (1) }
We set
$$
{\mathcal C}_p = \lambda(\H) \Delta^{\frac{1}{p}}.
$$
It follows from Lemmas \ref{2Inc} and \ref{2Dense} that this is a dense subspace of 
$L^p({\mathcal L}G)$. Moreover by the second half of (\ref{3HH}),
$\lambda(\phi f) \Delta^{\frac{1}{p}}$ belongs to 
$L^p({\mathcal L}G)$ for any $f\in\H$. We claim that 
$\phi$ is a bounded Fourier multiplier 
on $L^p({\mathcal L}G)$ if and only if there exists a constant $C\geq 0$ such that
\begin{equation}\label{3Alternate}
\bignorm{\lambda(\phi f)\Delta^{\frac{1}{p}}}_p\leq C
\bignorm{\lambda(f)\Delta^{\frac{1}{p}}}_p,\qquad f\in\H.
\end{equation}
The `if part' follows from the first half of Definition \ref{3BFM} and Remark
\ref{3Theta} applied with $\theta=0$.

To prove the `only if part', assume that 
$\phi$ is a bounded Fourier multiplier 
on $L^p({\mathcal L}G)$ and consider the resulting map
$M_{\phi,p}\colon L^p({\mathcal L}G)\rightarrow L^p({\mathcal L}G)$.
We need \cite[Theorem 4.5, (1)]{Terp2} which asserts that
\begin{equation}\label{3Inter}
\bignorm{\lambda(f)\Delta^{\frac{1}{p}}}_p\leq \norm{f}_{p'},
\qquad f\in L^1(G)\cap L^2(G).
\end{equation}
Here $p'$ is the conjugate number of $p$. 

Let $f\in \H$. 
There exists a compact subset $K\subset G$ and a  sequence $(g_n)_{n\geq 1}$ in 
$\H\star \H$ such that ${\rm Supp}(g_n)\subset K$ for all $n\geq 1$ and 
$\norm{g_n -f}_{p'}\to 0$. 
This implies that $\norm{\phi g_n -\phi f}_{p'}\to 0$. Then, applying (\ref{3Inter})
twice, we see that 
$$
\lambda(\phi f)\Delta^{\frac{1}{p}}=\lim_n
\lambda(\phi g_n)\Delta^{\frac{1}{p}} = \lim_n M_{\phi,p}\bigl(
\lambda(g_n)\Delta^{\frac{1}{p}} \bigr) = 
M_{\phi,p}\bigl(
\lambda(f)\Delta^{\frac{1}{p}} \bigr).
$$
This yields (\ref{3Alternate}), with $C=\norm{M_{\phi,p}}$.

\smallskip
{\bf (2)}   Arguing as in Remark \ref{3Theta}, we obtain that for any 
$\theta\in[0,1]$, 
$$
{\mathcal C}_p = \Delta^{\frac{\theta}{p}}\lambda(\H) \Delta^{\frac{(1-\theta)}{p}},
$$
and that $\phi$ is a bounded Fourier multiplier 
on $L^p({\mathcal L}G)$ if and only if there exists a constant $C\geq 0$ such that
$$
\bignorm{\Delta^{\frac{\theta}{p}}\lambda(\phi f)\Delta^{\frac{(1-\theta)}{p}}}_p\leq C
\bignorm{\Delta^{\frac{\theta}{p}}\lambda(f)\Delta^{\frac{(1-\theta)}{p}}}_p,\qquad f\in\H.
$$
Further in this case,
$$
M_{\phi,p}\Bigl(\Delta^{\frac{\theta}{p}} \lambda(f) \Delta^{\frac{(1-\theta)}{p}}\Bigr)
= \Delta^{\frac{\theta}{p}} \lambda(\phi f) \Delta^{\frac{(1-\theta)}{p}},
\qquad f\in\H.
$$
\end{remark}

We conclude this section with a duality result that extends \cite[Lemma 6.4]{AK}.
We set $\check{g}(t)=g(t^{-1})$ for any $g\in L^1(G)+L^\infty(G)$. With this notation,
$g^*=\Delta^{-1}\overline{\check{g}}$ for all $g\in L^1(G)$. 
It is plain that for any $g\in\H$, $\Delta^{-\frac12}\overline{\check{g}}\in\H$.
Consequently, $g^*\in\H$ and hence by part (1) of Lemma \ref{2Inc},
$\Delta^{\frac12}\lambda(g) = \bigl(\lambda(g^*)\Delta^\frac12\bigr)^*\in L^2({\mathcal L}G)$. We now claim that
\begin{equation}\label{3Trace}
{\rm Tr}\bigl(\Delta^\frac12\lambda(g)\lambda(f)\Delta^\frac12\bigr) = \int_G\Delta(t)^{-1} g(t^{-1})f(t)\, dt,\qquad
f,g\in \H.
\end{equation}
By the above discussion, both sides of this equality are well defined.
Then by polarization, it suffices to prove it when $g=f^*$. In this case,
both sides of (\ref{3Trace}) are equal to $\norm{f}_2^2$, by (\ref{2FP}). Hence
the desired equality holds.

\begin{proposition}\label{3Duality} 
Let $1<p<\infty$, let $p'$ denote its conjugate number and let $\phi\in L^\infty(G)$. If $\phi$
is a bounded Fourier multiplier on $L^p({\mathcal L}G)$, 
then $\check{\phi}$ is a bounded Fourier multiplier on $L^{p'}({\mathcal L}G)$,
and 
$
M_{\phi,p}^*= M_{\check{\phi},p'}.
$
\end{proposition}

\begin{proof} 
Assume first that $p\leq 2$ and let $T=M_{\phi,p}$. Let $g\in\H$
and $f\in\H\star\H$. 
By Remark
\ref{3Theta} and  (\ref{3Comm}), we have
\begin{align*}
\Bigl\langle T^*\bigl(\Delta^{\frac{1}{p'}}\lambda(g)\bigr),
\Delta^{\frac{1}{p} -\frac12}\lambda(f)\Delta^{\frac12}\Bigr\rangle
& = \Bigl\langle \Delta^{\frac{1}{p'}}\lambda(g) , T\bigl(\Delta^{\frac{1}{p} -\frac12}\lambda(f)\Delta^{\frac12}\bigr)\Bigr\rangle\\
& = \Bigl\langle  \Delta^{\frac{1}{p'}}\lambda(g) ,\Delta^{\frac{1}{p} -\frac12}\lambda(\phi
f)\Delta^{\frac12}\Bigr\rangle\\
& = {\rm Tr}\Bigl(\Delta^{\frac{1}{p'}}\lambda(g) 
\Delta^{\frac{1}{p} -\frac12}\lambda(\phi
f)\Delta^{\frac12}\Bigr)\\
& =  {\rm Tr}\Bigl(\Delta^{\frac{1}{2}}\lambda
\bigl(\Delta^{\frac12-\frac{1}{p}} g\bigr)
\lambda(\phi
f)\Delta^{\frac12}\Bigr).
\end{align*}
Hence by (\ref{3Trace}),
$$
\Bigl\langle T^*\bigl(\Delta^{\frac{1}{p'}}\lambda(g)\bigr),
\Delta^{\frac{1}{p} -\frac12}\lambda(f)\Delta^{\frac12}\Bigr\rangle 
= \int_G\Delta(t)^{-\frac{1}{p'}-\frac12}g(t^{-1})\phi(t) f(t)\, dt.
$$
Changing $g$ into $\check{\phi} g$ and $T$ into $I_{L^p}$, this identity implies that
$$
\Bigl\langle T^*\bigl(\Delta^{\frac{1}{p'}}\lambda(g)\bigr),
\Delta^{\frac{1}{p} -\frac12}\lambda(f)\Delta^{\frac12}\Bigr\rangle 
=\Bigl\langle\Delta^{\frac{1}{p'}}\lambda(\check{\phi} g),
\Delta^{\frac{1}{p} -\frac12}\lambda(f)\Delta^{\frac12}\Bigr\rangle.
$$
The result follows.

We now assume that $p>2$. This case is a bit more delicate. Again, we let $T=M_{\phi,p}$. 
This time, we fix $g\in\H\star \H$
and consider an arbitrary  $f\in \H$. 
By the discussion before (\ref{3Trace}),
$\Delta^{\frac12}\lambda\bigl(\check{\phi}g)$
belongs to $L^2({\mathcal L}G)$ and hence
$$
\Delta^{\frac12}\lambda\bigl(\check{\phi}g)\Delta^{\frac{1}{p'}-\frac12}\lambda(f)\Delta^{\frac{1}{p}}\,\in L^1({\mathcal L}G).
$$
Then arguing as in the first case, we obtain that 
$$
\Bigl\langle T^*\bigl(\Delta^{\frac12}\lambda(g)\Delta^{\frac{1}{p'}-\frac12}\bigr),
\lambda(f)\Delta^{\frac{1}{p}}\Bigr\rangle=
{\rm Tr}\Bigl(\Delta^{\frac12}\lambda\bigl(\check{\phi}g)\Delta^{\frac{1}{p'}-\frac12}\lambda(f)\Delta^{\frac{1}{p}}\Bigr).
$$
If we can prove that 
\begin{equation}\label{Goal}
\Delta^{\frac12}\lambda\bigl(\check{\phi}g)\Delta^{\frac{1}{p'}-\frac12}\in L^{p'}({\mathcal L}G),
\end{equation}
then we obtain the result as in the first case. At this stage, we have an estimate
$$
\Bigl\vert 
{\rm Tr}\Bigl(\Delta^{\frac12}\lambda\bigl(\check{\phi}g)
\Delta^{\frac{1}{p'}-\frac12}\lambda(f)\Delta^{\frac{1}{p}}\Bigr)\Bigr\vert
\lesssim \bignorm{\lambda(f)\Delta^{\frac{1}{p}}}_p.
$$
Then by duality, there exists $\Gamma\in L^{p'}({\mathcal L}G)$ such that 
$$
{\rm Tr}\Bigl(\Delta^{\frac12}\lambda\bigl(\check{\phi}g)\Delta^{\frac{1}{p'}-\frac12}\lambda(f)\Delta^{\frac{1}{p}}\Bigr)
={\rm Tr}\Bigl(\Gamma\lambda(f)\Delta^{\frac{1}{p}}\Bigr),\qquad f\in \H.
$$
Using (\ref{3HH}), we may change $f$ into $f\star \Delta^{\frac{1}{p}}h$  for any $h\in\H$ in the above identity.
Then using  (\ref{3Comm}), we write
$$
\lambda(f\star \Delta^{\frac{1}{p}}h)
\Delta^{\frac{1}{p}}
=\lambda(f)\lambda\bigl(\Delta^{\frac{1}{p}}h\bigr)\Delta^{\frac{1}{p}}
=\lambda(f)\Delta^{\frac{1}{p}}\lambda(h)
$$ 
and we obtain
$$
{\rm Tr}\Bigl(\Delta^{\frac12}\lambda\bigl(\check{\phi}g)
\Delta^{\frac{1}{p'}-\frac12}\lambda(f)\Delta^{\frac{1}{p}}\lambda(h)\Bigr)
={\rm Tr}\Bigl(\Gamma\lambda(f)\Delta^{\frac{1}{p}}\lambda(h)\Bigr),\qquad f,h\in \H.
$$
Since $C_c(G)\subset \H$ and $\lambda(C_c(G))$ is $w^*$-dense in ${\mathcal L}G$,
this implies
$$
\Delta^{\frac12}\lambda\bigl(\check{\phi}g)
\Delta^{\frac{1}{p'}-\frac12}\lambda(f)\Delta^{\frac{1}{p}}
=\Gamma\lambda(f)\Delta^{\frac{1}{p}},\qquad f\in\H.
$$
We deduce  that $\Delta^{\frac12}\lambda\bigl(\check{\phi}g)\Delta^{\frac{1}{p'}-\frac12}=\Gamma$, and (\ref{Goal}) follows.
\end{proof}

\medskip
\section{Extension properties}\label{4}

In this section, we leverage the Haagerup--Junge--Xu extension theorem from \cite{HJX} to transfer positivity and contractivity of Fourier multipliers from the von Neumann algebra level to non-commutative $L^p({\mathcal L}G)$-spaces. In particular, we show that continuous characters $\phi$ on $G$ give rise to positive onto isometric Fourier multipliers on all $L^p(\mathcal{L}G)$ (Corollary~\ref{4Easy}).

For any $1\leq p\leq \infty$, we let $L^p({\mathcal L}G;\varphi_0)$ denote
the Haagerup non-commutative $L^p$-space associated with $\varphi_0$,
see \cite{H, Terp0, HJX} or \cite[Chapter 9]{Hiai}. 
We recall that $L^p({\mathcal L}G;\varphi_0)$ consists of closed, densely defined
operators acting on $L^2(\Rdb;L^2(G))$. We let $\pi\colon
{\mathcal L}G\to B(L^2(\Rdb;L^2(G)))$ be the $*$-representation
such that $\pi({\mathcal L}G)=L^\infty({\mathcal L}G;\varphi_0)$.

Let $D$ be the so-called
density operator of $\varphi_0$ and recall $\M$ from (\ref{3M}).
Then as noticed in \cite[Remark 5.6]{HJX},
$$
{\mathcal L} : = D^\frac12 \pi(\M) D^\frac12 \subset L^1({\mathcal L}G;\varphi_0),
$$
and ${\mathcal L}$ is a dense subspace. This is an analogue of Lemma \ref{2Dense}, part  (2).

Let $T\colon {\mathcal L}G\to {\mathcal L}G$ be a positive map such that 
\begin{equation}\label{4L1}
\varphi_0(T(x))\leq\varphi_0(x),\qquad x\in {\mathcal L}G^+.
\end{equation}
Then $T$ maps $\M$ into $\M$. One may therefore define
$T^1\colon {\mathcal L}\to{\mathcal L}$ by 
setting 
$$
T^1\bigl(D^\frac12 \pi(x) D^\frac12) = D^\frac12 \pi\bigl(T(x)\bigr) D^\frac12 ,\qquad x\in \M.
$$
The Haagerup-Junge-Xu extension theorem
from \cite[Theorem 5.1 and Remark 5.6]{HJX} asserts that 
$T^1$ uniquely extends to a 
bounded operator 
$$
T^1\colon L^1({\mathcal L}G;\varphi_0)\longrightarrow L^1({\mathcal L}G;\varphi_0),
$$
with $\norm{T^1}\leq 1$.
This result can be translated into the framework considered in Section \ref{3}, as follows.
In the next statement, we rely on Lemma \ref{2Inc}, part (2).

\begin{theorem}\label{4Ext}
Let $T\colon {\mathcal L}G\to {\mathcal L}G$ be a positive map satisfying (\ref{4L1}).
Then the mapping 
$$
\Delta^{\frac12} \M\Delta^{\frac12}\longrightarrow 
\Delta^{\frac12} \M\Delta^{\frac12},\qquad 
\Delta^{\frac12} x\Delta^{\frac12}\,\mapsto\, \Delta^{\frac12} T(x)\Delta^{\frac12},
$$
uniquely extends to a contractive map $L^1({\mathcal L}G)\to 
L^1({\mathcal L}G)$.
\end{theorem}

\begin{proof}
We use the natural unitary identification between $L^2(\Rdb;L^2(G))$
and the Hilbertian tensor product $L^2(\Rdb)\overset{2}{\otimes} L^2(G)$.
We will use tensor products of unbounded operators acting on 
$L^2(\Rdb)$ and $L^2(G)$, respectively, for which we refer to \cite[Remark 11.33]{Hiai}.
In this context, we have $D=d\otimes 1$, where $d$  is the analytic generator of the unitary group of
translation operators on $L^2(\Rdb)$. Apply \cite[pp. 204-206]{Hiai} to the left Plancherel weight $\varphi_0$ on $M={\mathcal L}G$ and to
$\psi=\psi_0$, the right Plancherel weight on $M'={\mathcal R}G$.
By (\ref{2Standard}), the operator 
$d_0$ defined on the first line of \cite[p. 204]{Hiai} 
coincides with $\Delta$. 
Let $U_0\colon L^2(\Rdb;L^2(G))\to L^2(\Rdb;L^2(G))$ be the unitary operator
considered in \cite[Lemma 11.30]{Hiai}. Then the aforementioned  lemma yields
\begin{equation}\label{4Hiai1}
U_0^*(1\otimes x)U_0=\pi(x),\qquad x\in {\mathcal L}G,
\end{equation}
and 
\begin{equation}\label{4Hiai2}
U_0^*(d\otimes\Delta)U_0 =D.
\end{equation}
By \cite[Lemma 11.32]{Hiai}, there exists an isometric isomorphism  
$$
W\colon L^1({\mathcal L}G)\longrightarrow 
L^1({\mathcal L}G;\varphi_0),\qquad W(a)= U_0^*(d\otimes a)U_0.
$$
 
For any $x\in\M$, we have
\begin{align*}
U_0^*\bigl(d\otimes \Delta^{\frac{1}{2}} x\Delta^{\frac{1}{2}}\bigr)U_0
& = 
U_0^*\Bigl(\bigl(d^{\frac{1}{2}}\otimes \Delta^{\frac{1}{2}}\bigr)(1\otimes x)
\bigl(d^{\frac{1}{2}}\otimes \Delta^{\frac{1}{2}}\bigr)\Bigr)U_0\\
& = 
U_0^*\bigl(d^{\frac{1}{2}}\otimes \Delta^{\frac{1}{2}}\bigr)U_0\cdotp U_0^*
(1\otimes x)U_0\cdotp U_0^*\bigl(d^{\frac{1}{2}}\otimes \Delta^{\frac{1}{2}}\bigr)U_0\\
& = 
\bigl(U_0^*(d\otimes \Delta)U_0\bigr)^{\frac{1}{2}} \cdotp U_0^*
(1\otimes x)U_0\cdotp \bigl(U_0^*(d\otimes \Delta)U_0\bigr)^{\frac{1}{2}}.
\end{align*}
Using (\ref{4Hiai1}) and (\ref{4Hiai2}), we deduce that 
$$
W\bigl(\Delta^\frac12 x\Delta^\frac12\bigr)  =
D^\frac12 x D^\frac12,\qquad x\in\M.
$$
Therefore, the contraction $W^{-1}T^1W$ maps $\Delta^\frac12 x\Delta^\frac12$
to $\Delta^\frac12 T(x)\Delta^\frac12$ for all $x\in\M$.
This proves the result.
\end{proof}

We refer to \cite{JX, LMZ2} for more information on the  Haagerup-Junge-Xu extension theorem.

For any $\phi\in C_b(G)$, let 
$M_\phi\colon\lambda(G)\to\lambda(G)$ be the linear map taking 
$\lambda(t)$  to $\phi(t)\lambda(t)$ for all $t\in G$. Following \cite{DCH}, we
say that $\phi$ is a bounded Fourier multiplier on ${\mathcal L}G$ if
there exists a (necessarily unique) $w^*$-continuous map ${\mathcal L}G\to {\mathcal L}G$
extending $M_\phi$.  In this case, we still denote this extension by
$M_\phi\colon{\mathcal L}G\to {\mathcal L}G$.

In the sequel, we let $e$ denote the unit of $G$. 

\begin{corollary}\label{4ExtFourier}
Let $\phi\in C_b(G)$ such that $\phi$ is a bounded Fourier multiplier on ${\mathcal L}G$
and $M_\phi$ is a positive map.
Then for any $1\leq p<\infty$, $\phi$ is a bounded Fourier multiplier on $L^p({\mathcal L}G)$, and we have
\begin{equation}\label{4Inter}
\norm{M_{\phi,p}\colon L^p({\mathcal L}G)\longrightarrow L^p({\mathcal L}G)}\leq \phi(e).
\end{equation}
\end{corollary}

\begin{proof}
The positivity assumption ensures that $\phi(e)\geq 0$ and 
$$
\bignorm{M_\phi\colon {\mathcal L}G\longrightarrow {\mathcal L}G} =\norm{M_\phi(1)}=
\phi(e).
$$
For any $f\in C_c(G)$ such that $\lambda(f)\geq 0$, we have
$\varphi_0(\lambda(f))= f(e)$, see e.g. \cite[§18.17]{Stra}.
Hence
$\varphi_0\bigl(M_\phi(\lambda(f))\bigr)= \phi(e) f(e)$ for any such $f$.
According to the Pedersen-Takesaki theorem \cite[Theorem 6.2]{Stra}, this implies that 
$
\varphi_0\circ M_\phi  =\phi(e)\varphi_0.
$
Hence by Theorem \ref{4Ext}, $\phi$ is a bounded Fourier multiplier on $L^1({\mathcal L}G)$, with
$$
\bignorm{M_{\phi,1}\colon L^1({\mathcal L}G)\longrightarrow L^1({\mathcal L}G)}\leq\phi(e).
$$
Once this result is settled, the boundedness on $L^p({\mathcal L}G)$ and the estimate (\ref{4Inter})
follow from Terp's interpolation theorem \cite[Section 2]{Terp1}.
\end{proof}

\begin{corollary}\label{4Easy}
Let $\phi$ be a continuous character on $G$.
For any $1\leq p<\infty$, $\phi$ is a bounded  Fourier multiplier on $L^p({\mathcal L}G)$
and $M_{\phi,p}$ is a positive onto isometry on $L^p({\mathcal L}G)$.
\end{corollary}

\begin{proof}
Let $\phi\colon G\to \Cdb$ be a continuous character. Then $\phi$
is positive definite and $\phi(e)=1$. Hence according to \cite[Proposition 4.2]{DCH},
$\phi$ is a bounded Fourier multiplier on ${\mathcal L}G$
and $M_\phi$ is a positive, contractive  map.
Hence by Corollary \ref{4ExtFourier}, $\phi$ is a bounded Fourier multiplier on $L^p({\mathcal L}G)$
and $M_{\phi,p}$ is a contraction. 

Since $\phi^{-1}=\overline{\phi}$  is also   a continuous character,
$M_{\phi,p}$ is invertible and $M_{\phi,p}^{-1} = M_{\overline{\phi},p}$ is a contraction. Consequently,
$M_{\phi,p}$ is an onto isometry.

The positivity of $M_{\phi,p}$ directly follows from the positivity of $M_\phi$ 
and from the symmetric nature of the definition of $M_{\phi,p}$.
\end{proof}

\begin{remark}\label{4cp}
Let $L^p(M):=L^p(M,H,\psi)$ be any Connes-Hilsum non-commutative 
$L^p$-space associated with 
a von Neumann algebra $M\subset B(H)$ and a normal semifinite
faithful weight $\psi$ on $M'$. For any integer $n\geq 1$, let 
$M_n$ be the $C^*$-algebra of all $n\times n$ matrices and
consider $M_n(M)\subset B(\ell^2_n\overset{2}{\otimes} H)$ in the usual way. 
Then the commutant of $M_n(M)$ is $I_{\ell^2_n}\otimes M'$.
We may therefore consider $\psi$ as a weight on $M_n(M)'$.
This allows to define the Connes-Hilsum space 
$$
L^p(M_n(M)):=L^p \bigl(M_n(M),\ell^2_n\overset{2}{\otimes} H,\psi\bigr). 
$$
It is well known that $L^p(M_n(M))$ can be regarded as a space of $n\times n$ matrices
with entries in $L^p(M)$, which allows to write an algebraic identification 
$$
L^p(M_n(M)) = M_n\otimes L^p(M).
$$
For any $S\colon  L^p(M)\to  L^p(M)$, we set $S_n=I_{M_n}\otimes S$  and we say
that $S$ is $n$-positive if the map $S_n\colon L^p(M_n(M))\to L^p(M_n(M))$ is positive. Next, we say that
$S$ is completely positive if $S$ is $n$-positive for all $n\geq 1$.

It is not hard to check (left to the reader) that in
Corollary \ref{4ExtFourier}, if $M_\phi\colon {\mathcal L}G\to  {\mathcal L}G$ is completely positive, 
then for any $1\leq p<\infty$, $M_{\phi,p}$ is completely positive on $L^p({\mathcal L}G)$.

Now coming back to Corollary \ref{4Easy}, we know from  \cite[Proposition 4.2]{DCH}
that if $\phi$ is a continuous character on $G$, then $M_\phi\colon {\mathcal L}G\to  {\mathcal L}G$ is completely positive.
Therefore, $M_{\phi,p}$ is completely positive for all $1\leq p<\infty$.
\end{remark}

We now turn to another application of the Haagerup-Junge-Xu theorem to Fourier multipliers.

\begin{proposition}\label{4Antipode}
Let $\phi\in L^\infty(G)$ and let $1\leq p<\infty$. If $\phi$ is a bounded Fourier multiplier on
$L^p({\mathcal L}G)$, then $\check{\phi}$ is also a bounded Fourier multiplier, with
$\norm{M_{\check{\phi},p}}= \norm{M_{\phi,p}}$.
\end{proposition}

\begin{proof}
We consider the anti-$*$-automorphism $\kappa\colon {\mathcal L}G\to {\mathcal L}G$ such
that $\kappa(\lambda(t))=\lambda(t^{-1})$ for all $t\in G$
(see \cite[Theorem 3.3.6]{EnockS}). 
This is a positive, involutive isometry.
For any $f\in L^1(G)$, we have
$$
\kappa(\lambda(f))=\int_G f(t)\lambda(t^{-1})\, dt\, = \int_G\Delta(t)^{-1}f(t^{-1})\lambda(t)\,dt,
$$
and hence
\begin{equation}\label{4Kappa}
\kappa(\lambda(f))=\lambda\bigl(\overline{f}^*\bigr),\qquad f\in L^1(G).
\end{equation}
Consequently,
$$
\varphi_0\circ\kappa\bigl(\lambda(f)^*\lambda(f)\bigr)
=\varphi_0\bigl(\kappa(\lambda(f))\kappa(\lambda(f^*))\bigr)
=\varphi_0\bigl(\lambda(\overline{f})^*\lambda(\overline{f})\bigr),\qquad f\in L^1(G)\cap L^2(G).
$$
Applying (\ref{2FP}), we deduce that $\varphi_0\circ\kappa$ and $\varphi_0$ coincide
on $\lambda(f)^*\lambda(f)$ for any  $f\in L^1(G)\cap L^2(G)$.  
Hence $\varphi_0\circ\kappa=\varphi_0$ by  the 
Pedersen-Takesaki theorem \cite[Theorem 6.2]{Stra}.

Therefore, we can apply  Theorem \ref{4Ext} to $\kappa$.
Combining  this extension result with Terp's interpolation theorem as in the proof of 
Corollary \ref{4ExtFourier}, and applying (\ref{4Kappa}), we deduce the existence of an involutive isometry
$\kappa_p\colon L^p({\mathcal L}G)\to L^p({\mathcal L}G)$ such that
$$
\kappa_p\bigl(\Delta^{\frac{1}{2p}} \lambda (f) \Delta^{\frac{1}{2p}}\bigr)
=\Delta^{\frac{1}{2p}}\lambda\bigl(\overline{f}^*\bigr) \Delta^{\frac{1}{2p}},\qquad
f\in\H\star \H.
$$
Assume that $\phi$ is a bounded Fourier multiplier on
$L^p({\mathcal L}G)$. A simple calculation shows that 
$$
\bigl(\kappa_p\circ M_{\phi,p}\circ\kappa_p\bigr)\bigl(\Delta^{\frac{1}{2p}} \lambda (f) \Delta^{\frac{1}{2p}}\bigr)
=\bigl(\Delta^{\frac{1}{2p}} \lambda (\check{\phi} f) \Delta^{\frac{1}{2p}}\bigr)\qquad
f\in\H\star \H.
$$
This implies that $\check{\phi}$ is a bounded Fourier multiplier and that 
$M_{\check{\phi},p}=\kappa_p\circ M_{\phi,p}\circ \kappa_p.$
\end{proof}

\begin{corollary}\label{5SelfDual}
Let $\phi\in L^\infty(G)$, let $1< p<\infty$ and let $p'$ be its conjugate number. 
Then $\phi$ is a bounded Fourier multiplier on
$L^p({\mathcal L}G)$ if and only if $\phi$ is a bounded Fourier multiplier on
$L^{p'}({\mathcal L}G)$. In this case, we
have $\norm{M_{\phi,p'}}=\norm{M_{\phi,p}}$ and $\norm{\phi}_\infty\leq \norm{M_{\phi,p}}$.
\end{corollary}

\begin{proof}
Assume that $\phi\in L^\infty(G)$
is a bounded Fourier multiplier on
$L^p({\mathcal L}G)$.
Combining Proposition \ref{4Antipode} and Proposition \ref{3Duality}, we obtain
that $\phi$ is a bounded Fourier multiplier on
$L^{p'}({\mathcal L}G)$, with $\norm{M_{\phi,p'}}=\norm{M_{\phi,p}}$. 
Then by interpolation,
$\norm{M_{\phi,2}}\leq \norm{M_{\phi,p}}$ and hence 
$\norm{\phi}_\infty\leq \norm{M_{\phi,p}}$
by Corollary \ref{3Always}.
\end{proof}

\begin{corollary}\label{5Bar}
Let $\phi\in L^\infty(G)$ and let $1< p<\infty$.
If $\phi$ is a bounded Fourier multiplier on
$L^p({\mathcal L}G)$, then $\overline{\phi}$ is also 
a bounded Fourier multiplier, with 
$\norm{M_{\overline{\phi},p}} = \norm{M_{\phi,p}}$.
\end{corollary}

\begin{proof}
According to Corollary \ref{5SelfDual}, we may assume that
$p\geq 2$. This allows us to apply part (2) of Remark
\ref{3p-geq-2}. Let $C=\norm{M_{\phi,p}}$. Then
$$
\bignorm{\Delta^{\frac{1}{2p}}\lambda(\phi f)\Delta^{\frac{1}{2p}}}_p\leq C
\bignorm{\Delta^{\frac{1}{2p}}\lambda(f)\Delta^{\frac{1}{2p}}}_p,\qquad f\in\H.
$$
Passing to adjoints in $L^p({\mathcal L}G)$, 
this yields
\begin{equation}\label{5adj}
\bignorm{\Delta^{\frac{1}{2p}}\lambda(\phi f)^*\Delta^{\frac{1}{2p}}}_p\leq C
\bignorm{\Delta^{\frac{1}{2p}}\lambda(f)^*\Delta^{\frac{1}{2p}}}_p,\qquad f\in\H.
\end{equation}
We note that
$$
(\phi f)^*(t) = \Delta(t)^{-1}\overline{\phi(t^{-1})}
\overline{f(t^{-1})} = \overline{\phi(t^{-1})} f^*(t).
$$
Thus, we have $\lambda(\phi f)^* = \lambda\bigl(\check{\overline{\phi}}f^*\bigr)$, whereas
$\lambda(f)^* = \lambda(f^*)$. Since $\H$ is closed under taking adjoints,
(\ref{5adj}) actually shows that
$\check{\overline{\phi}}$ is a bounded 
Fourier multiplier on $L^p({\mathcal L}G)$, with norm
less than $C$. Combining with Proposition \ref{4Antipode}, 
we deduce that $\overline{\phi}$ is a bounded 
Fourier multiplier on $L^p({\mathcal L}G)$, with norm
less than $C$. The result follows.
\end{proof}

\medskip
\section{A characterization of positive isometric Fourier multipliers}\label{5}

The main objective of this section is Theorem \ref{5Positive} which provides a converse 
to Corollary \ref{4Easy} for $1<p\not=2<\infty$. The special cases $p=1$ and $p=2$ are treated in the 
next section.

We recall that for any locally measurable functions $\phi_1,\phi_2\colon G\to\Cdb$, saying that 
$\phi_1=\phi_2$ locally almost everywhere means that the set
$\{\phi_1\not=\phi_2\}$ is a locally null set. If $\phi_1,\phi_2$ are 
bounded except on a locally null set, the above is equivalent to
the equality $\phi_1=\phi_2$ in $L^\infty(G)$.

\begin{lemma}\label{5Continuous}
Let $1\leq p<\infty$ and let $\phi\in L^\infty(G)$ be a bounded Fourier 
multiplier on $L^p({\mathcal L}G)$. If  the resulting operator
$M_{\phi,p}\colon L^p({\mathcal L}G)\to L^p({\mathcal L}G)$ is positive, then
for every compact set $K\subset G$, the restriction $\phi_{\vert K}$ is almost everywhere equal 
to a continuous function $K\to\Cdb$.
\end{lemma}

\begin{proof}
We adapt the proof of \cite[Lemma 6.10]{AK}. 
For any arbitrary
$g\in C_c(G)$, the product $\Delta^{\frac{1}{2p}}
\lambda(g^* \star g)\Delta^{\frac{1}{2p}}=\Delta^{\frac{1}{2p}}
\lambda(g)^*\lambda(g)\Delta^{\frac{1}{2p}}$
is a positive element of $L^p({\mathcal L}G)$. 
Hence by the positivity assumption on $M_{\phi,p}$, 
$$
\Delta^{\frac{1}{2p}}\lambda\bigl(\phi(g^* \star g)\bigr)\Delta^{\frac{1}{2p}} 
\,\in L^p({\mathcal L}G)^+.
$$

Take any $\xi\in C_c(G)$. 
Then $\Delta^{\frac{1}{2p}}\xi\in C_c(G)$. Since 
$\lambda(\phi(g^* \star g))\Delta^{\frac{1}{2p}}\xi=\phi(g^* \star g)\star \Delta^{\frac{1}{2p}}\xi$,
this implies that $\lambda(\phi(g^* \star g))\Delta^{\frac{1}{2p}}\xi$ belongs 
to $C_c(G)$. Hence $\xi$ belongs to the domain 
of the operator $\Delta^{\frac{1}{2p}}
\lambda\bigl(\phi(g^* \star g)\bigr)\Delta^{\frac{1}{2p}}$.
Let $(\,\cdotp\vert\,\cdotp)$ denote the inner product on $L^2(G)$. 
Then it follows from the above that
$$
\bigl(\Delta^{\frac{1}{2p}}
\lambda\bigl(\phi(g^* \star g)\bigr)\Delta^{\frac{1}{2p}}\xi\,\vert\,\xi\bigr)\geq 0.
$$
Therefore,
$$
\bigl(\lambda(\phi(g^* \star g)) \Delta^{\frac{1}{2p}}\xi\,
\big\vert\,\Delta^{\frac{1}{2p}}\xi\bigr) \geq 0.
$$
Changing $\xi$ into $\Delta^{-\frac{1}{2p}}\xi$,
we deduce that 
\begin{equation}\label{Pospos}
\bigl(\lambda(\phi(g^*\star g)) \xi\,
\big\vert\,\xi\bigr) \geq 0,
\end{equation}
for all $\xi\in C_c(G)$. By density, this implies that (\ref{Pospos}) 
holds for all $\xi\in L^2(G)$.

For any $g\in C_c(G)$, we write
$$
(g^*\star g)(v) = \int_G \Delta(u^{-1})\overline{g(u^{-1})} g(u^{-1}v)\, du,\qquad v\in G.
$$
Then, we have 
$$
\lambda(\phi(g^* \star g)) = \int_{G^2}\Delta(u^{-1})\overline{g(u^{-1})} 
g(u^{-1}v)\phi(v)\lambda(v)\, dudv.
$$
Changing $u$ into $t^{-1}$ and then $v$ into $t^{-1}s$, we deduce that
$$
\lambda(\phi(g^* \star g)) = \int_{G^2} \overline{g(t)}g(s)
\phi(t^{-1}s) \lambda(t^{-1}s)\,dtds.
$$
Therefore,
$$
\bigl(\lambda(\phi(g^* \star g)) \xi\,
\big\vert\,\xi\bigr) = 
\int_{G^2} \overline{g(t)}g(s)
\phi(t^{-1}s) \bigl(\lambda(t^{-1}s) 
\xi\,
\vert\, \xi\bigr)  \,dtds, 
$$
for all $\xi\in L^2(G).$
This integral is non-negative, by (\ref{Pospos}). 
We deduce that for all $\xi\in L^2(G)$,  the function
$t\mapsto \phi(t)(\lambda(t)\xi\vert\xi)$ is positive definite
in the sense of \cite[Definition VII.3.20]{Tak2}. According to 
\cite[Proposition VII.3.21]{Tak2} and its proof, this implies that 
this function belongs to the Fourier-Stieltjes algebra $B(G)$, in the sense that
it is locally almost everywhere equal to an element of $B(G)$.
By polarization, for all $\xi,\xi'\in L^2(G)$, we deduce the same property for the function 
$t\mapsto \phi(t)(\lambda(t)\xi\vert\xi')$. Thus, for all $\gamma$ in the Fourier algebra $A(G) = \{\, t \mapsto ( 
\lambda(t)\,\xi\vert \xi' ) \mid \xi, \xi' \in L^2(G) \,\}$,
the product $\phi\gamma$ is locally almost everywhere 
equal to a continuous function.
Since for any compact set $K\subset G$, there exists $\gamma\in A(G)$ that does not vanish on 
$K$, the result follows.
\end{proof}

The next lemma  is a consequence of Sherman's work
concerning  isometries
on non-commutative $L^p$-spaces \cite{Sh} (see also \cite{Sh2}).

Recall that a Jordan homomorphism $J\colon M\to N$ between two
von Neumann algebras $M$ and $N$  is a bounded linear map which satisfies
\begin{equation}\label{5Jordan}
J(xy+yx)=J(x)J(y)+ J(y)J(x)
\qquad\hbox{and}\qquad
J(x^*)=J(x)^*,
\end{equation}
for all $x,y\in M$. If further $J$ is a bijection, we say that 
$J$ is a Jordan isomorphism.

It follows from \cite[Theorem 3.3]{Sto} that $J$ is a Jordan isomorphism if and only if there exist
von Neumann  algebra decompositions 
\begin{equation}\label{5Decomp1}
M=M_1\oplus M_2
\qquad\hbox{and}\qquad
N=N_1\oplus N_2, 
\end{equation}
as well as a
$*$-isomorphism $\pi\colon M_1\to N_1$ and an anti-$*$-isomorphism
$\rho\colon M_2\to N_2$, such that 
\begin{equation}\label{5Decomp2}
J=\pi\oplus\rho.
\end{equation}
This implies that $J$ is $w^*$-continuous. We refer to \cite{HOS} and \cite[Exercices 10.5.21-10.5.31]{KR} 
for more information on Jordan homomorphisms.

\begin{lemma}\label{5J}
Let $1<p\not=2<\infty$ and let $M,N$ be two von Neumann algebras. Let
$L^p(M)$ and $L^p(N)$ be any Connes-Hilsum spaces built upon $M$ and $N$, respectively.
Let $S\colon L^p(M)\to L^p(N)$ be an onto positive isometry.
Then there exists a Jordan isomorphism $J\colon M\to N$
such that 
\begin{equation}\label{5Deriv}
S(ax+xa) =S(a)J(x) + J(x)S(a),\qquad a\in L^p(M),\  x\in M.
\end{equation}  
\end{lemma}

\begin{proof}
Let $S\colon L^p(M)\to L^p(N)$ be an onto positive isometry.
Sherman's theorem \cite[Theorem 1.2]{Sh} asserts that there exists 
a Jordan isomorphism $J\colon M\to N$
such that 
\begin{equation}\label{5Facto}
S(a)= \bigl(J^{-1}_*(a^p)\bigr)^{\frac{1}{p}},\qquad a\in L^p(M)^+.
\end{equation}
In this statement, $J_*\colon L^1(N)\to L^1(M)$ is the pre-adjoint of $J$,
taking into account the natural isometric isomorphisms
$M_*\simeq L^1(M)$ and $N_*\simeq L^1(N)$, and 
$$
J^{-1}_*\colon L^1(M)\longrightarrow L^1(N)
$$
is its inverse. In \cite{Sh,Sh2}, this statement is given within the framework
of Haagerup non-commutative $L^p$-spaces. However using the isomorphism
between the Haagerup spaces and the Connes-Hilsum spaces (see
\cite[Chapter IV]{Terp0} or \cite[Section 11.3]{Hiai}), we see that 
Sherman's proof works as well in the Connes-Hilsum framework.

We are going to show that the mapping $J$ from (\ref{5Facto}) provides
the identity (\ref{5Deriv}).
It follows from the proof of \cite[Theorem 1.2]{Sh} that $M,N$ and $J$
have decompositions (\ref{5Decomp1}) and (\ref{5Decomp2}) such that
\begin{equation}\label{5Direct1}
S(ax) = S(a)\pi(x),\qquad a\in L^p(M),\ x\in M_1,
\end{equation}
and
\begin{equation}\label{5Direct2}
S(xa) = \rho(x)S(a),\qquad a\in L^p(M),\ x\in M_2.
\end{equation}
Since $S$ is positive, we have $S(b)=S(b^*)^*$ for all 
$b\in L^p(M)$. Hence for all $a\in L^p(M)$ and
all $x\in M_1$, we  have 
$$
S(xa) = S(a^*x^*)^* = \bigl(S(a^*)J(x^*)\bigr)^*= J(x^*)^*S(a^*)^*= J(x)S(a),
$$
by (\ref{5Direct1}). Applying once more  the latter identity, we obtain (\ref{5Deriv}) in
the case when $x\in M_1$. Now using  (\ref{5Direct2}) instead of  (\ref{5Direct1}),
we obtain  (\ref{5Deriv}) in
the case when $x\in M_2$. This establishes the identity  (\ref{5Deriv}) for all $x\in M$.
\end{proof}

\begin{theorem}\label{5Positive} 
Let $\phi\in L^\infty(G)$ and let
$1<p\not=2<\infty$. Assume that $\phi$ is a bounded Fourier multiplier on  $L^p({\mathcal L}G)$
and that the resulting map
$M_{\phi,p}\colon L^p({\mathcal L}G)\to L^p({\mathcal L}G)$ is a positive, onto isometry.
Then $\phi$ is locally almost everywhere equal to a continuous 
character.
\end{theorem}

\begin{proof}
We assume for simplicity that $G$ is $\sigma$-compact and refer to the last lines 
of this proof for the general case.
Thus, ``locally almost everywhere" just means ``almost everywhere" and by 
Lemma \ref{5Continuous}, $\phi$ is almost everywhere equal to a continuous function. 
We may therefore assume that $\phi\in C_b(G)$.

In summary, to establish the multiplicativity of $\phi$, we apply the identity~\eqref{5Deriv} to carefully chosen approximate units in ${\mathcal L}G$ and pass to a   limit, obtaining a pointwise relation that yields the desired conclusion, while also proving some auxiliary identities that simplify the computations.

For any $t\in G$ and for any $f\in C_c(G)$, we define 
$\lambda'(t)f\in C_c(G)$ by
$$
[\lambda'(t)f](s)=f(st^{-1}),\qquad s\in G.
$$
We observe that 
\begin{equation}\label{5Modular}
\Delta(t)\lambda(f)\lambda(t) = \lambda\bigl(\lambda'(t)f\bigr),
\qquad f\in C_c(G),\ t\in G.
\end{equation}
Indeed, 
$$
\lambda(\lambda'(t)f) = \int_G f(st^{-1})\lambda(s)\, ds= \Delta(t) \int_G f(s)\lambda(st)\, ds
 = \Delta(t) \Bigl(\int_G f(s)\lambda(s)\, ds\Bigr)\lambda(t).
$$

Consider a net $(f_i)_{i\in I}$ of $C_c(G)\star C_c(G)$ such that $f_i\geq 0$ and $\int_G f_i(t)\, dt=1$
for all $i\in I$,  the support of every
$f_i$ is contained in some compact neighborhood
$V_i$ of $e$, the net $(V_i)_{i\in I}$ is decreasing, and the intersection of the $V_i$ is equal to $\{e\}$. 
Then, we let $e_i=\lambda(f_i)$ for every $i\in I$.

 Let 
$C(G)$ be the space of continuous functions on $G$. We first observe
that 
\begin{equation}\label{5SOT}
\hbox{SOT-}\lim_i \lambda(hf_i) =
h(e),\qquad h\in C(G).
\end{equation}
Indeed, given any $h\in C(G)$ and $\xi\in L^2(G)$, we may write
\begin{align*}
\lambda(hf_i)\xi - h(e)\xi
&=\int_G f_i(t) \bigl(h(t)\lambda(t)\xi -h(e)\xi\bigr)\, dt\\
&=\int_G f_i(t) \bigl(h(t)-h(e)\bigr)
\lambda(t)\xi\, dt\, +h(e)\int_G
f_i(t)\bigl(\lambda(t)\xi-\xi\bigr)\, dt.
\end{align*}
This implies
$$
\bignorm{\lambda(hf_i)\xi - h(e)\xi}_2\leq \norm{\xi}_2 \int_G f_i(t) \bigl\vert h(t)-h(e)\bigr\vert\,
 dt\, +\vert h(e)\vert\int_G
f_i(t)\norm{\lambda(t)\xi-\xi}_2\, dt.
$$
These two integrals tend to $0$ when $i\to\infty$, which proves (\ref{5SOT}).
It is plain that property  (\ref{5SOT}) holds with $f_i^*$ instead of $f_i$.
Applying this with $h\equiv 1$, we obtain in particular that 
\begin{equation}\label{5ei}
\hbox{SOT-}\lim_i e_i=1
\qquad\hbox{and}\qquad
\hbox{SOT-}\lim_i e_i^*=1.
\end{equation}

Let $S=M_{\phi,p} \colon L^p({\mathcal L}G)\to L^p({\mathcal L}G)$ and let 
$J$ be the  Jordan isomorphism provided by Lemma \ref{5J}.
Let $s\in G$. We are going to apply the identity 
(\ref{5Deriv}) with 
\begin{equation}\label{xa}
x=e_i\lambda(s)
\qquad\hbox{and}\qquad
a=e_i\Delta^{\frac{1}{p}}.
\end{equation}
The latter belongs to $L^p({\mathcal L}G)$, by Remark  \ref{3Theta}.
We note that by (\ref{3Comm}) and (\ref{5Modular}), we 
can write alternatively
\begin{equation}\label{xa2}
a=\Delta^{\frac{1}{p}}\lambda(\Delta^{-\frac{1}{p}}f_i)
\qquad\hbox{and}\qquad
x=\Delta(s)^{-1}\lambda(\lambda'(s)f_i).
\end{equation}
In the following, we use either
(\ref{xa}) or (\ref{xa2})
to express $a$ and $x$.

First, we write
$$
ax= \Delta(s)^{-1}\Delta^{\frac{1}{p}}\lambda(\Delta^{-\frac{1}{p}}f_i) 
\lambda(\lambda'(s)f_i)
= \Delta(s)^{-1} \Delta^{\frac{1}{p}}
\lambda\Bigl((\Delta^{-\frac{1}{p}}f_i) \star (\lambda'(s)f_i)\Bigr).
$$
Then by Remark  \ref{3Theta} again, we have 
$$
S(ax) = \Delta(s)^{-1} \Delta^{\frac{1}{p}}
\lambda\Bigl(\phi\bigl(
(\Delta^{-\frac{1}{p}}f_i) \star (\lambda'(s)f_i)\bigr)\Bigr).
$$
Likewise,
$$
xa = \Delta(s)^{-1}\lambda\Bigl((\lambda'(s)f_i) \star 
f_i\Bigr)\Delta^{\frac{1}{p}},
$$
and hence
$$
S(xa) = \Delta(s)^{-1}\lambda\Bigl(\phi\bigl((\lambda'(s)f_i) \star 
f_i\bigr)\Bigr)\Delta^{\frac{1}{p}}
$$
We also have $S(a)=\Delta^{\frac{1}{p}}
\lambda(\phi\Delta^{-\frac{1}{p}}f_i)$, and then
$$
S(a)J(x) = \Delta^{\frac{1}{p}}
\lambda(\phi\Delta^{-\frac{1}{p}}f_i)
J(e_i\lambda(s)).
$$
Likewise,  $S(a)= \lambda(\phi f_i)\Delta^{\frac{1}{p}}$ and we have
$$
J(x)S(a) = J(e_i\lambda(s))\lambda(\phi f_i)\Delta^{\frac{1}{p}}.
$$

For any integer $n\geq 1$, we let 
$E_n=\{t\in G\, :\, n^{-1}\leq\Delta(t)\leq n\}$. We write 
(\ref{5Deriv}) using the above computations, 
and we multiply the resulting identity on the left and on the right 
by the indicator function $\chi_{E_n}$. We obtain the identity
\begin{align}\label{5in}
\Delta(s)^{-1}\Bigl[
\chi_{E_n}\Delta^{\frac{1}{p}} & 
\lambda\Bigl(\phi\bigl(
(\Delta^{-\frac{1}{p}}f_i)\star (\lambda'(s)f_i) \bigr)\Bigr)\chi_{E_n}
+\chi_{E_n} \lambda\Bigl(\phi\bigl((\lambda'(s)f_i)\star 
f_i\bigr)\Bigr)\Delta^{\frac{1}{p}}\chi_{E_n} \Bigr]\\
&= \chi_{E_n} \Delta^{\frac{1}{p}}\lambda(\phi\Delta^{-\frac{1}{p}}f_i)
J(e_i\lambda(s))\chi_{E_n} + 
\chi_{E_n} J(e_i\lambda(s))\lambda(\phi f_i)\Delta^{\frac{1}{p}}\chi_{E_n},\nonumber
\end{align}
valid for all $i\in I$ and all $n\geq 1$.
The function $\Delta^{\frac{1}{p}}\chi_{E_n}$ is bounded, and therefore
the above identity holds in $B(L^2(G))$.

We momentarily fix $n\geq 1$. We let $i\to\infty$ and unless otherwise stated,
we use SOT-convergence  in the following limiting processes. We note that by
(\ref{5ei}), we both have $e_i\lambda(s)\to\lambda(s)$ and $(e_i\lambda(s))^* 
= \lambda(s)^*e_i^*\to\lambda(s)^*$.
Since $J$ is $w^*$-continuous, it therefore follows from \cite[Lemma 2.2]{AKLZ} that  
$$
J(e_i\lambda(s))\longrightarrow J(\lambda(s)).
$$
Next, using the continuity
of $\phi$ and the fact that $\Delta(e)=1$, 
we derive from (\ref{5SOT}) that 
$$
\lambda(\phi f_i)\longrightarrow \phi(e)
\qquad\hbox{and}\qquad
\lambda(\phi 
\Delta^{-\frac{1}{p}} f_i)\longrightarrow \phi(e).
$$
Thus, we know the SOT-limit of the right hand side of (\ref{5in}).

Passing to the left hand side, we write
$$
\bigl(\lambda'(s)f_i) \star f_i\bigr)(t)  = \int_G [\lambda'(s) f_i](u) f_i(u^{-1}t) \, du
\ = \int_G f_i(us^{-1}) f_i(u^{-1}t) \, du,
$$
for any $t\in G$. We deduce
\begin{align*}
\lambda\Bigl(\phi\bigl((\lambda'(s)f_i) \star 
f_i\bigr)\Bigr) & =
\int_{G^2}\phi(t) f_i(us^{-1}) f_i(u^{-1}t) \lambda(t)\, dudt\\
& = \int_{G^2}\phi(ut) f_i(us^{-1}) f_i(t) \lambda(ut)\, dtdu\\
& = \Delta(s) \int_{G^2}\phi(ust) f_i(u) f_i(t) \lambda(ust)\, dtdu.
\end{align*}
Likewise,
$$
\phi(s)\lambda\bigl((\lambda'(s)f_i) \star 
f_i\bigr) =
\Delta(s) \int_{G^2}\phi(s) f_i(u) f_i(t) \lambda(ust)\, dtdu,
$$
hence
\begin{align*}
\Bignorm{\lambda\Bigl(\phi\bigl((\lambda'(s)f_i)\star 
f_i\bigr)\Bigr) -\phi(s)& \lambda\bigl((\lambda'(s)f_i)\star 
f_i\bigr)}_{B(L^2(G))}\\
&\leq \Delta(s)
\int_{G^2}\vert\phi(ust)-\phi(s)\vert
f_i(u)f_i(t)\, dtdu.
\end{align*}
We deduce, by the continuity of $\phi$ and the properties of the net $(f_i)_{i\in I}$, that
\begin{equation}\label{5NormCV}
\Bignorm{\lambda\Bigl(\phi\bigl((\lambda'(s)f_i)\star 
f_i\bigr)\Bigr) -\phi(s)\lambda\bigl((\lambda'(s)f_i)\star 
f_i\bigr)}_{B(L^2(G))}\longrightarrow 0,
\end{equation}
when $i\to\infty$. 
By (\ref{5Modular}), we have
$$
\lambda\bigl((\lambda'(s)f_i)\star 
f_i\bigr)  = \lambda(\lambda'(s)f_i)\lambda(f_i)
= \Delta(s) \lambda(f_i)\lambda(s)\lambda(f_i)=\Delta(s) e_i\lambda(s)e_i.
$$
Since $(e_i)_{i\in I}$ is bounded, we deduce from (\ref{5ei}) that 
$$
\lambda\bigl((\lambda'(s)f_i)\star 
f_i\bigr)  \longrightarrow \Delta(s)\lambda(s).
$$
Combining with  (\ref{5NormCV}), we obtain that
$$
\lambda\Bigl(\phi\bigl((\lambda'(s)f_i)\star 
f_i\bigr)\Bigr)\longrightarrow \Delta(s)\phi(s)\lambda(s).
$$
A similar calculation yields
$$
\lambda\Bigl(\phi\bigl((\Delta^{-\frac{1}{p}}
f_i)\star (\lambda'(s)f_i)\bigr)\Bigr)\longrightarrow \Delta(s)\phi(s)\lambda(s).
$$

Passing to the limit in (\ref{5in}), we deduce from these calculations that  
\begin{align*}
\chi_{E_n}\Delta^{\frac{1}{p}} \phi(s)\lambda(s)\chi_{E_n} &
+ \chi_{E_n}\phi(s)\lambda(s)\Delta^{\frac{1}{p}}\chi_{E_n} \\
& =\chi_{E_n}\Delta^{\frac{1}{p}}\phi(e)J(\lambda(s))\chi_{E_n}
+ \chi_{E_n}\phi(e)J(\lambda(s))
\Delta^{\frac{1}{p}}\chi_{E_n}.
\end{align*}
Let $A,X\in B(L^2(G))$ be defined by 
$$
A=\chi_{E_n}\Delta^{\frac{1}{p}}\qquad\hbox{and}\qquad 
X=\chi_{E_n}\bigl(\phi(s)\lambda(s) -\phi(e)J(\lambda(s))\bigr)\chi_{E_n}.
$$
Since 
$$
\chi_{E_n}^2=\chi_{E_n}
\qquad\hbox{and}\qquad
\chi_{E_n}\Delta^{\frac{1}{p}}=\chi_{E_n}\Delta^{\frac{1}{p}}\chi_{E_n}
=\Delta^{\frac{1}{p}}\chi_{E_n},
$$
the above identity reads
$$
AX+XA=0.
$$
According to the definition of $E_n$, $\sigma(A)\subset [n^{-\frac{1}{p}},n^\frac{1}{p}]$, 
hence
$\sigma(A)\cap \sigma(-A)=\emptyset$. It therefore follows from \cite[Theorem VII.2.1]{BR}
that $X=0$.
Thus,
$$
\chi_{E_n}\bigl(\phi(s)\lambda(s) -\phi(e)J(\lambda(s))\bigr)\chi_{E_n}=0,\qquad n\geq 1.
$$

Since $(E_n)_{n\geq 1}$ is non-decreasing and the union of the $E_n$ is equal to $G$, $\chi_{E_n}\to 1$ in the SOT.
Hence we deduce from above that
$$
\phi(e)J(\lambda(s)) = \phi(s)\lambda(s).
$$
With this identity in hands, the last part of the proof of \cite[Theorem 3.7]{AKLZ} 
shows that $\overline{\phi(e)}\phi$ is a character. Let $\widetilde{\phi} = \overline{\phi(e)}\phi$, then 
$M_{\widetilde{\phi},p}$ is a positive isometry, by Corollary \ref{4Easy}. Since 
$M_{\widetilde{\phi},p} =\overline{\phi(e)}M_{\phi,p}$ and 
$M_{\phi,p}$ is a positive isometry,  we have
$\phi(e)=1$ and therefore,
$\phi=\widetilde{\phi}$ is a character. This completes the proof in the $\sigma$-compact case.

If $G$ is not $\sigma$-compact it suffices to apply the proof of \cite[Corollary 3.13]{AKLZ}
to deduce the result from the above arguments. Details are left to the reader.
\end{proof}

We end this section with a result showing that in the case when
$p\geq 2$, the onto assumption in Theorem \ref{5Positive} is superfluous. We do not know 
if the same holds for $p<2$.

\begin{proposition}\label{5Onto}
Let $\phi\in L^\infty(G)$ and let $2\leq p<\infty$. Assume that $\phi$ is a bounded Fourier multiplier on $L^p({\mathcal L}G)$, and that the resulting operator $M_{\phi,p} \colon L^p({\mathcal L}G)\to L^p({\mathcal L}G)$ is an isometry.
Then $M_{\phi,p}$ is onto.
\end{proposition}

\begin{proof}
We adapt the proof of \cite[Lemma 3.9]{AKLZ}. We assume that
$M_{\phi,p}\colon L^p({\mathcal L}G)\to L^p({\mathcal L}G)$ is an isometry.
In particular, $M_{\phi,p}$ is injective.

Let $K\subset G$ be an arbitrary  compact 
set and let $N_K=\{t\in K\, :\, \phi(t)=0\}$.
The indicator function
$h=\chi_{N_K}$ belongs to  $L^1(G)\cap L^2(G)$, hence $\lambda(h)$ to $\N$.
By Lemma \ref{2Inc}, (1) and Remark \ref{3p-geq-2}, $\lambda(h)
\Delta^{\frac{1}{p}}$ belongs to $L^p({\mathcal L}G)$
and $M_{\phi,p}\bigl(\lambda(h)\Delta^{\frac{1}{p}}\bigr) = \lambda(\phi h)\Delta^{\frac{1}{p}}$.
Since $\phi h=0$, this means that $M_{\phi,p}\bigl(\lambda(h)\Delta^{\frac{1}{p}}\bigr)=0$.
This implies  $\lambda(h)=0$, and hence $h=0$. Thus,
\begin{equation}\label{5=0}
\vert N_K\vert =0.
\end{equation}

Let $f\in C_c(G)$ and let
$K={\rm Supp}(f)$. 
For any $\delta>0$, set
$N_\delta =\{t\in K\, :\, \vert\phi(t)\vert <\delta\}$. 
Then it follows from (\ref{5=0}) that 
\begin{equation}\label{Lim}
\lim_{\delta\to 0}\vert N_\delta\vert \,=0.
\end{equation}
We may define 
$\phi^{-1} f\chi_{N_\delta^c}$, which belongs to $L^1(G)\cap L^\infty(G)$. Hence arguing as above, 
$$
\lambda\bigl(\phi^{-1} f\chi_{N_\delta^c}\bigr)\Delta^{\frac{1}{p}}\in L^p({\mathcal L}G)
\qquad\hbox{and}\qquad
M_\phi\Bigl(\lambda\bigl(\phi^{-1} f\chi_{N_\delta^c}\bigr)\Delta^{\frac{1}{p}} \Bigr)
= \lambda\bigl(f\chi_{N_\delta^c}\bigr) 
\Delta^{\frac{1}{p}}.
$$
Consequently,
$$
M_{\phi,p}\Bigl(\lambda\bigl(\phi^{-1} f\chi_{N_\delta^c}\bigr)
\Delta^{\frac{1}{p}} \Bigr)
- \lambda(f) \Delta^{\frac{1}{p}}  = 
-\lambda\bigl(f\chi_{N_\delta}\bigr) \Delta^{\frac{1}{p}}.
$$
By (\ref{3Inter}), this implies that 
$$
\Bignorm{M_{\phi,p}\Bigl(\lambda\bigl(\phi^{-1} f\chi_{N_\delta^c}\bigr)\Delta^{\frac{1}{p}}\Bigr)
- \lambda(f) \Delta^{\frac{1}{p}}}_p \leq \norm{f\chi_{N_\delta}}_{p'}
\leq \norm{f}_\infty \vert N_\delta\vert^{\frac{1}{p'}}.
$$
According to (\ref{Lim}), this implies that 
$\lambda(f) \Delta^{\frac{1}{p}} \in \overline{{\rm Ran}(M_{\phi,p})}.$ By the density Lemma \ref{2Dense}, (1), this implies that the range of $M_{\phi,p}$ is dense.
Since $M_{\phi,p}$ is an isometry, this means that $M_{\phi,p}$ is onto.
\end{proof}

\medskip
\section{The cases $p=1$ and $p=2$}\label{6}
In this section we consider the endpoint cases $p=1$ and $p=2$. We first show that Theorem~\ref{5Positive} extends to the case $p=1$ without assuming positivity, as follows.

\begin{proposition}\label{6p1}
Let $\phi\in L^\infty(G)$ be a bounded Fourier multiplier on $L^1({\mathcal L}G)$. 
If the operator 
$M_{\phi,1}\colon L^1({\mathcal L}G)\to L^1({\mathcal L}G)$
is an onto isometry, then there exists a complex number $\delta$ with $\vert\delta\vert =1$ 
such that $\delta\phi$ is locally almost everywhere equal to a continuous 
character.
\end{proposition}

\begin{proof}
As in the proof of Theorem \ref{5Positive}, we may assume 
for simplicity that $G$ is $\sigma$-compact.
Let $T=M_{\phi,1}^*\colon {\mathcal L}G\to {\mathcal L}G$. The argument in the proof of 
Proposition \ref{3Duality} shows that
$T(\lambda(g))=\lambda(\check{\phi}g)$ for all $g\in\H$.
This implies that 
$$
\norm{\lambda(\check{\phi} g)}\leq \norm{\lambda(g)},\qquad g\in L^1(G).
$$
Then,  \cite[Proposition 1.2, (3)$\Rightarrow$(2)]{DCH} and its
proof show that
the function $\phi$ belongs to $C_b(G)$, that $\check{\phi}$ is a bounded
Fourier multiplier on ${\mathcal L}G$ and that $T=M_{\check{\phi}}$.

Since $M_{\phi,1}$  is an onto isometry, its adjoint
$T$ is an onto isometry as well. Hence by Kadison's non-commutative
Banach-Stone theorem from \cite{Kad}, 
there exist a Jordan 
isomorphism $J\colon {\mathcal L}G\to {\mathcal L}G$
and a unitary $U\in {\mathcal L}G$ such that 
$$
T(x) = UJ(x),\qquad x\in {\mathcal L}G.
$$

Consider an arbitrary $s\in G$.
Let $(e_i)_{i\in I}$ be the net introduced in the proof of Theorem 
\ref{5Positive}. 
It follows from \cite[Lemma 3.11]{AKLZ} (which does not use 
any unimodularity assumption) that the following SOT-convergences hold:
\begin{equation}\label{6-311}
T(\lambda(s)e_i)\to \check{\phi}(s)\lambda(s),\quad 
T(\lambda(s)e_i^2)\to \check{\phi}(s)\lambda(s)\quad\hbox{and}\quad
T\bigl(e_i\lambda(s)e_i\bigr) \to\check{\phi}(s)\lambda(s).
\end{equation}
Moreover it follows from an observation in the proof of  Theorem 
\ref{5Positive} that
\begin{equation}\label{6-J}
J(\lambda(s)e_i)\longrightarrow J(\lambda(s)).
\end{equation}

By the first half of (\ref{5Jordan}), we have 
$$
J\bigl(e_i\lambda(s)e_i + \lambda(s)e_i^2\bigr) 
=J(e_i)J\bigl(\lambda(s)e_i\bigr) + J\bigl(\lambda(s)e_i\bigr) J(e_i).
$$
After left multiplication by $U$, this means that
$$
T\bigl(e_i\lambda(s)e_i\bigr) +T\bigl(\lambda(s)e_i^2\bigr) 
=T(e_i)J\bigl(\lambda(s)e_i\bigr) + T\bigl(\lambda(s)e_i\bigr) J(e_i).
$$
Passing to the SOT-limit when $i\to \infty$ and applying (\ref{6-311}) and (\ref{6-J}), we deduce that 
$\check{\phi}(s)\lambda(s) +\check{\phi}(s)\lambda(s) =
\check{\phi}(e)J(\lambda(s)) + \check{\phi}(s)\lambda(s)$, that is,
$$
\check{\phi}(e) J(\lambda(s))=\check{\phi}(s)\lambda(s),\qquad s\in G.
$$
We can now conclude as in the proof of Theorem \ref{5Positive} that 
$\check{\phi}(e)^{-1}\check{\phi}$ is  a character and the result follows at once.
\end{proof}

\begin{remark}\label{6-Infty}
Let $\phi\in C_b(G)$ such that $\phi$
is a bounded Fourier multiplier on 
${\mathcal L}G$ and assume that $M_\phi\colon
{\mathcal L}G\to {\mathcal L}G$ is an isometry.
(See the paragraph preceding Corollary \ref{4Inter}
for definitions.) Then $\vert\phi(t)\vert = \norm{M_\phi(\lambda(t))}=1$
for all $t\in G$, hence $\lambda(G)\subset {\rm Ran}(M_\phi)$. Since ${\rm Ran}(M_\phi)$ is
$w^*$-closed, this implies that $M_\phi$ is onto. 
Then, the proof of Proposition \ref{6p1} shows as well
the existence of a complex number $\delta$
with $\vert\delta\vert =1$ such that $\delta\phi$ is a character. 
\end{remark}

We now turn to the case $p=2$.
Until Lemma \ref{6Pos}, we  consider any Connes-Hilsum
non-commutative $L^2$-space $L^2(M)$ built upon
an arbitrary von Neumann algebra $M$. In
\cite{Ju}, Junge introduced $\ell^1$-valued 
non-commutative spaces $L^p(M;\ell^1)$. Here, we consider
$L^2(M; \ell^1_2)$ obtained by restricting to the case 
$p=2$ and to the situation when $\ell^1$ is replaced by the 
2-dimensional space $\ell^1_2$. Thus, $L^2(M; \ell^1_2)$ is the
direct sum
$L^2(M)\oplus L^2(M)$ equipped with 
a specific norm. We refer e.g. to \cite{LMZ} for information on this norm
(that we will not explicitly use here).

We say that two
elements $a,b\in L^2(M)$ are disjoint if $a^*b=0$ and $ab^*=0$.
Then, we say that a bounded map $S\colon L^2(M)\to L^2(M)$ is separating 
if, for all disjoint $a$ and $b$ in $L^2(M)$, 
the images $S(a)$ and $S(b)$ are disjoint.
This terminology extends the one in  \cite[Section 3]{LMZ}
to the possibly non tracial case.

\begin{lemma}\label{6LMZ}
\begin{itemize} 
\item [(1)] Let $a,b\in L^2(M)$. Then, $a$ and $b$ are disjoint if and only if 
$$
\bignorm{(a,b)}_{L^2(M; \ell^1_2)}\leq \bigl(\norm{a}_2^2 +\norm{b}_2^2\bigr)^\frac12.
$$
\item [(2)] Let $S\colon L^2(M)\to L^2(M)$ be a 2-positive isometry. Then $S$ is separating.
\end{itemize}
\end{lemma}

\begin{proof}
The assertion (1) is proved as \cite[Lemma 4.1]{LMZ} in the case when $L^2(M)$ is a tracial 
non-commutative $L^2$-space. The proof extends verbatim to the general case.

Now let $S\colon L^2(M)\to L^2(M)$ be a 2-positive isometry. Then by
\cite[Proposition 5.1]{LMZ}, we have
\begin{equation}\label{6Cont}
\norm{(S(a),S(b))}_{L^2(M; \ell^1_2)}\leq\bignorm{(a,b)}_{L^2(M; \ell^1_2)}.
\end{equation}
Again this result is proved for tracial spaces only but it extends to the general case.

Let $a,b\in L^2(M)$ be disjoint elements. Then combining part (1) and (\ref{6Cont}), we obtain
$$
\norm{(S(a),S(b))}_{L^2(M; \ell^1_2)}\leq\bigl(\norm{a}_2^2 +\norm{b}_2^2\bigr)^\frac12.
$$
Since $\norm{S(a)}_2=\norm{a}_2$ and $\norm{S(b)}_2=\norm{b}_2$, this shows, by part (1) again, that
$S(a)$ and $S(b)$ are disjoint.
\end{proof}

\begin{lemma}\label{6Pos}
Let $S\colon L^2(M)\to L^2(M)$ be a positive onto isometry. Then 
$S^{-1}$ is positive.
\end{lemma}

\begin{proof}
An element $a\in L^2(M)$ is positive if and only if 
$(a\vert b)\geq 0$ for all $b\in L^2(M)^+$. Using this criterion,
we see that $S^*$ is positive if $S\colon L^2(M)\to L^2(M)$ is positive.
If further $S$ is an onto isometry, then 
$S^{-1}=S^*$, hence $S^{-1}$ is positive.
\end{proof}

Recall from Remark \ref{4cp} that for any
$S\colon L^2(M)\to L^2(M)$, we let 
$$
S_2 \colon L^2(M_2(M))\longrightarrow
L^2(M_2(M))
$$
be its natural tensor extension.
The following is an analogue of Theorem \ref{5Positive} for $p=2$. 
This is a converse to the fact that if $\phi$ is a continuous 
character on $G$, then $M_{\phi,2}$ is a completely isometric onto isometry
(see Corollary \ref{4Easy} and Remark \ref{4cp}).

\begin{theorem}\label{6p=2}
Let $\phi\in L^\infty(G)$. If the operator 
$S:=M_{\phi,2}\colon L^2({\mathcal L}G)\to L^2({\mathcal L}G)$
is such that  $S_2$ is a positive isometry, then $\phi$ is locally 
almost everywhere equal to a continuous character.
\end{theorem}

\begin{proof}
Assume that for $S=M_{\phi,2}$, the operator
$S_2$ is a  positive isometry. Then $S$ is a $2$-positive
isometry hence by Lemma \ref{6LMZ}, 
$S$ is separating. By Proposition \ref{5Onto}, $S$ is onto, which implies that $S_2$
is onto. Then, applying Lemma \ref{6Pos} with $S_2$, we obtain that
$S_2^{-1}$ is positive. Thus, $S^{-1}$ is  a 2-positive isometry
hence the above reasoning shows  that $S^{-1}$ is 
separating as well. Now a thorough look at  Sherman's proof 
reveals that \cite[Theorem 1.2]{Sh}
 holds for any onto isometry between non-commutative
$L^2$-spaces that is separating, as well as its inverse.
Then the proofs of Lemma \ref{5J} and Theorem \ref{5Positive} apply
to   $S=M_{\phi,2}$
if $S_2$ is a positive isometry.
\end{proof}

We do not know if Theorem \ref{5Positive} fully extends to $p=2$.

\bigskip\noindent
{\bf Acknowledgements:}
The last two authors thank the Centre International de Rencontres Mathématiques (CIRM) for fostering an exceptional research environment that greatly accelerated and facilitated this work through a RiR grant. The third author gratefully acknowledges the University  Marie and Louis Pasteur for providing an excellent research environment during her visit.
\vskip 1cm

\bibliographystyle{abbrv}

\end{document}